\documentclass{article}

\usepackage{microtype}
\usepackage{graphicx}
\usepackage{subfigure}
\usepackage{booktabs} 
\usepackage{hyperref}

\usepackage[accepted]{icml2021}

\usepackage{amsmath,amssymb,amsthm,graphicx,mathrsfs,url}

\usepackage{enumitem}
\usepackage[english]{babel}
\usepackage[utf8]{inputenc}     
\usepackage[T1]{fontenc}         
\usepackage{lmodern}             
\usepackage{mathtools}
\usepackage{bbm,bm}
\usepackage{enumitem}
\numberwithin{equation}{section}
\usepackage{url}
\usepackage[capitalize,noabbrev]{cleveref}

\graphicspath{{fig/}}

\newcommand{\R}{\mathbb{R}}
\newcommand{\N}{\mathbb{N}}

\newcommand{\DD}{\mathcal{D}}
\newcommand{\PP}{\mathcal{P}}
\newcommand{\QQ}{\mathcal{Q}}

\newcommand{\E}{\mathbb{E}}

\newcommand{\MM}{\mathcal{M}}
\newcommand{\CC}{\mathcal{C}}
\newcommand{\FF}{\mathcal{F}}

\renewcommand{\SS}{\mathcal{S}}

\newcommand{\dd}{\mathrm{d}}
\newcommand{\dgm}{\mathrm{Dgm}}

\newcommand{\eps}{\varepsilon}

\newcommand{\Pers}{\mathrm{Pers}}  
\newcommand{\diam}{\mathrm{diam}}
\newcommand{\upperdiag}{\Omega} 
\newcommand{\groundspace}{\overline{\upperdiag}} 

\newcommand{\thediag}{{\partial \Omega}} 

\newcommand{\supp}{\mathrm{spt}}

\newcommand{\dist}{\mathrm{dist}}

\newcommand{\defeq}{\vcentcolon=}

\newcommand{\projdiag}{\mathrm{proj}_\thediag}

\newcommand{\Adm}{\mathrm{Adm}}

\newcommand{\id}{\mathrm{id}}
\renewcommand{\epsilon}{\varepsilon}

\newcommand{\bc}{\mathbf{c}}
\newcommand{\bC}{\mathbf{C}}

\newcommand{\ones}{\mathbf{1}}

\newcommand{\EPD}{\mathbf{E}(P)} 
\renewcommand{\epsilon}{\varepsilon}

\newcommand{\OT}{\mathrm{OT}}

\newcommand{\Dp}{\mathrm{OT}_p}

\newcommand{\mmin}{m_\mathrm{min}}
\newcommand{\mmax}{m_\mathrm{max}}
\newcommand{\Dmin}{D_\mathrm{min}}

\newcommand{\p}[1]{\left(#1 \right)}

\DeclareMathOperator*{\argmin}{arg\,min}

\newcommand{\Cech}{\v Cech}

\newtheorem{definition}{Definition}
\newtheorem{theorem}{Theorem}
\theoremstyle{plain}
\newtheorem{prop}[theorem]{Proposition} 
\newtheorem{lemma}{Lemma}
\newtheorem{corollary}{Corollary}
\newtheorem{remark}{Remark}
\newtheorem*{theorem*}{Theorem}

\icmltitlerunning{Estimation and quantization of EPDs}  

\begin{document}

\twocolumn[
\icmltitle{Estimation and Quantization of Expected Persistence Diagrams}



\icmlsetsymbol{equal}{*}

\begin{icmlauthorlist}
\icmlauthor{Vincent Divol}{equal,inria,lmo}
\icmlauthor{Théo Lacombe}{equal,inria}
\end{icmlauthorlist}

\icmlaffiliation{inria}{DataShape, Inria Saclay, France}
\icmlaffiliation{lmo}{Laboratoire de Mathématiques d'Orsay, Université Paris-Sud, France}

\icmlcorrespondingauthor{Vincent Divol}{vincent.divol@inria.fr}
\icmlcorrespondingauthor{Théo Lacombe}{theo.lacombe@inria.fr}

\icmlkeywords{Optimal Transport, Topological Data Analysis, Algorithms, Quantization, Statistical learning.}

\vskip 0.3in
]



\printAffiliationsAndNotice{\icmlEqualContribution} 

\begin{abstract}
Persistence diagrams (PDs) are the most common descriptors used to encode the topology of structured data appearing in challenging learning tasks;~think e.g.~of graphs, time series or point clouds sampled close to a manifold.
Given random objects and the corresponding distribution of PDs, one may want to build a statistical summary---such as a mean---of these random PDs, which is however not a trivial task as the natural geometry of the space of PDs is not linear. 
In this article, we study two such summaries, the Expected Persistence Diagram (EPD), and its quantization. The EPD is a measure supported on $\R^2$, which may be approximated by its empirical counterpart. We prove that this estimator is optimal from a minimax standpoint on a large class of models with a parametric rate of convergence. The empirical EPD is simple and efficient to compute, but possibly has a very large support, hindering its use in practice. To overcome this issue, we propose an algorithm to compute a quantization of the empirical EPD, a measure with small support which is shown to approximate with near-optimal rates a quantization of the theoretical EPD.
\end{abstract}

\vspace{-.5cm}
\section{Introduction}
\vspace{-.1cm}
Topological data analysis (TDA) is a modern field in data science which has found a variety of succesful domains of application such as material science \citep{tda:saadatfar2017pore,tda:buchet2018persistent}, cellular data \citep{tda:camara2017topological}, social graph classification \citep{tda:zhao2019learning,tda:carriere2019perslay}, shape analysis \citep{tda:li2014persistence,tda:carriere2015stableSignature3DShape} to name a few. It provides a machinery to encode the topological properties (such as the presence of connected components, loops, cavities, etc.) of a structured object in a multi-scale fashion. Relying on persistent homology theory \citep{tda:edelsbrunner2000topologicalSeminal, tda:zomorodian2005computing, tda:edelsbrunner2010computational}, its main output is a descriptor called a \emph{persistence diagram} (PD): it is a discrete measure $\sum_{i\in I} \delta_{x_i}$ (roughly, a set of points) supported on the open half-plane $\upperdiag =\{(t_1,t_2)\in\R^2,\ t_2 > t_1 \}$, where each point $x_i$ of the PD accounts in a quantitative way for the presence of a topological feature in a given object. 
The space of PDs, $\DD$, is equipped with an \emph{optimal partial transport} metric $\Dp$, where $1 \leq p \leq \infty$, which shares similarities with the so-called Wasserstein metric $W_p$ used in the optimal transport literature \citep{ot:villani2008optimal,otam}. 

\vspace{-.1cm}
\textbf{Statistics with PDs.}
In applications, one is generally led to consider a sample of several PDs, say $\mu_1, \dots, \mu_n$, encoding the topology of some underlying phenomenon generating the different observations. 
Assuming that these PDs are sampled i.i.d.~according to some underlying distribution $P$, it is natural to search for some characteristic quantities to describe $P$. 
As the space of PDs $(\DD, \Dp)$ is not a vector space, but only a metric space, even building elementary statistics is a difficult task. For instance, approximating Fr\'echet means (a.k.a.~barycenters) of a sample of PDs with respect to $\Dp$ metrics requires to develop specific techniques \citep{tda:turner2014frechet, tda:lacombe2018large, tda:vidal2019progressive}, while their exact computation is intractable.
An alternative is to embed the PDs in a Hilbert or Banach space, using explicit vectorizations \citep{tda:bubenik2015statistical, tda:adams2017persistenceImages} or implicit through kernel methods \citep{tda:reininghaus2015stable, tda:carriere2017sliced}, then using standard statistical and learning tools. 
However, such embeddings do not preserve the metric structure of the space of PDs  \citep{tda:bubenik2019embeddings,tda:wagner2019nonembeddability} nor the interpretability of PDs. 
In comparison, the \emph{expected persistence diagram} (EPD) $\mathbf{E}(P)$ of a distribution $P$ of PDs lies in a natural  metric extension of the space of PDs while its empirical counterpart can be computed faithfully. 
Originally introduced in \citep{tda:divol2018density}, the EPD is a measure on $\upperdiag$ which associates to each set $A \subset \upperdiag$ the expected number of points which belongs to $A$ in the random diagrams $\mu\sim P$. 
The properties of this  object were studied in \citep{tda:divol2018density,tda:divol2019understanding}.

\begin{figure*}
	\center
	\includegraphics[width=0.12\textwidth]{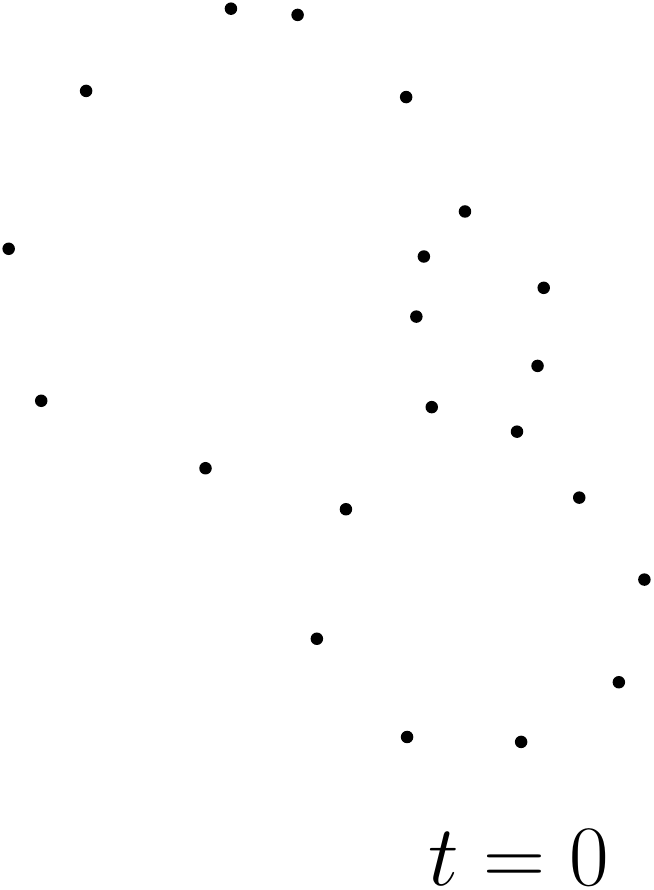}
	\includegraphics[width=0.12\textwidth]{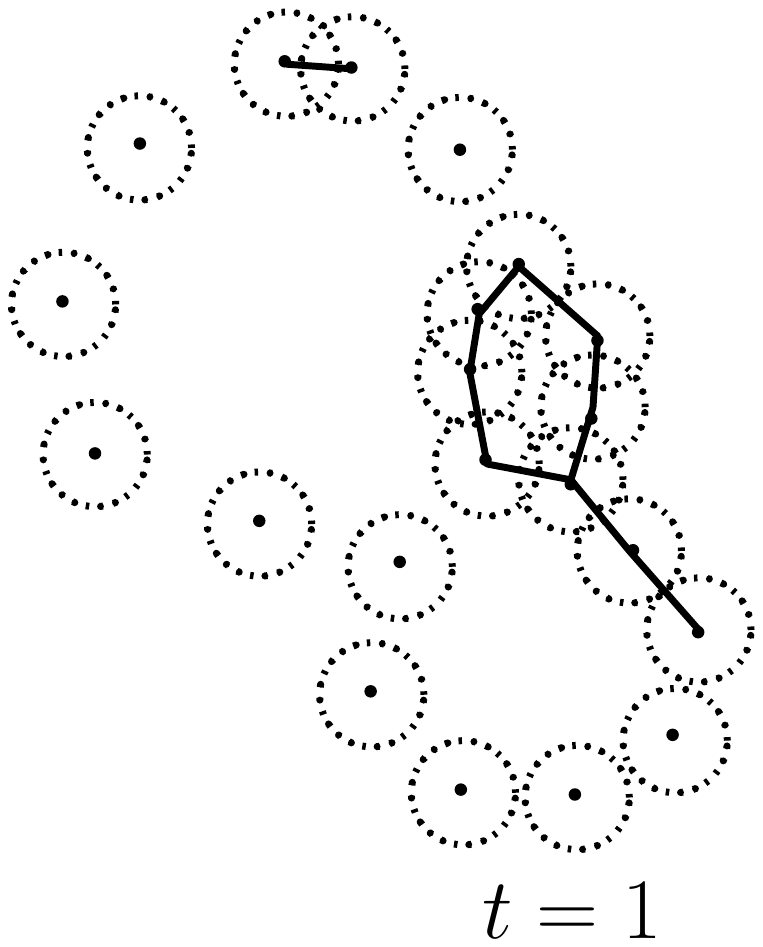}
	\includegraphics[width=0.12\textwidth]{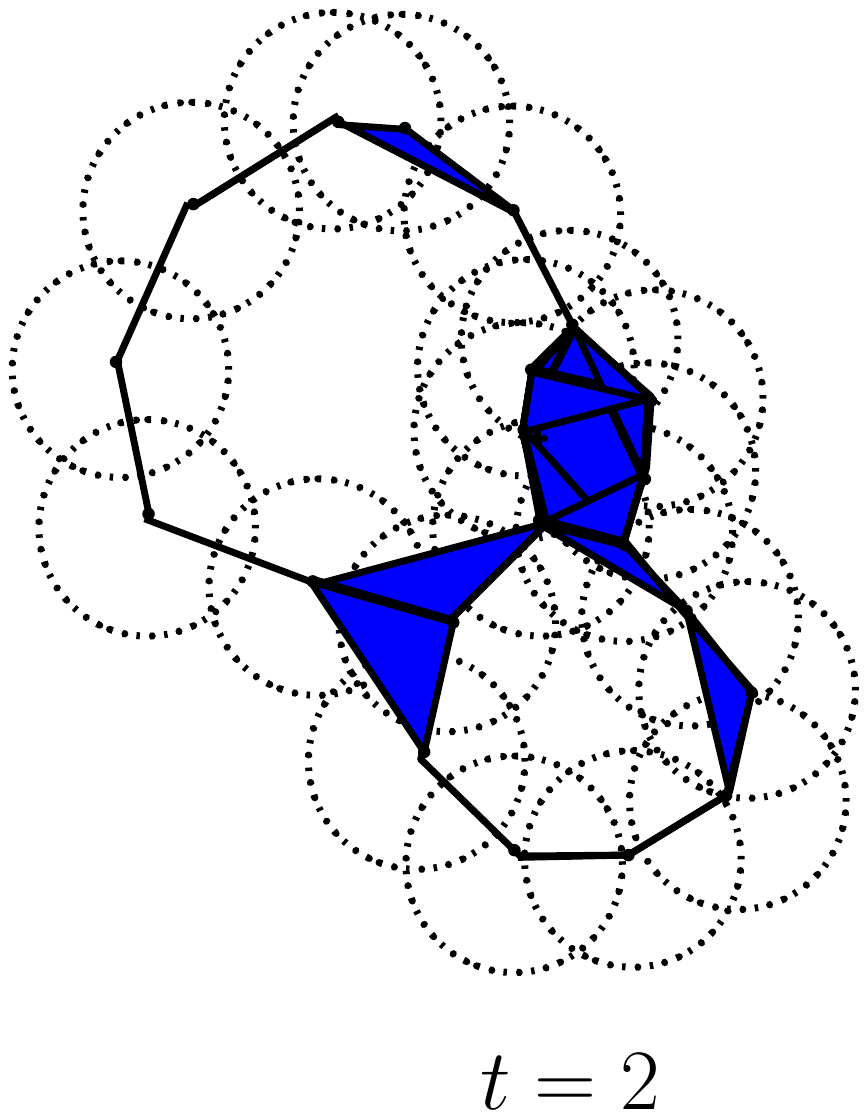}
	\includegraphics[width=0.12\textwidth]{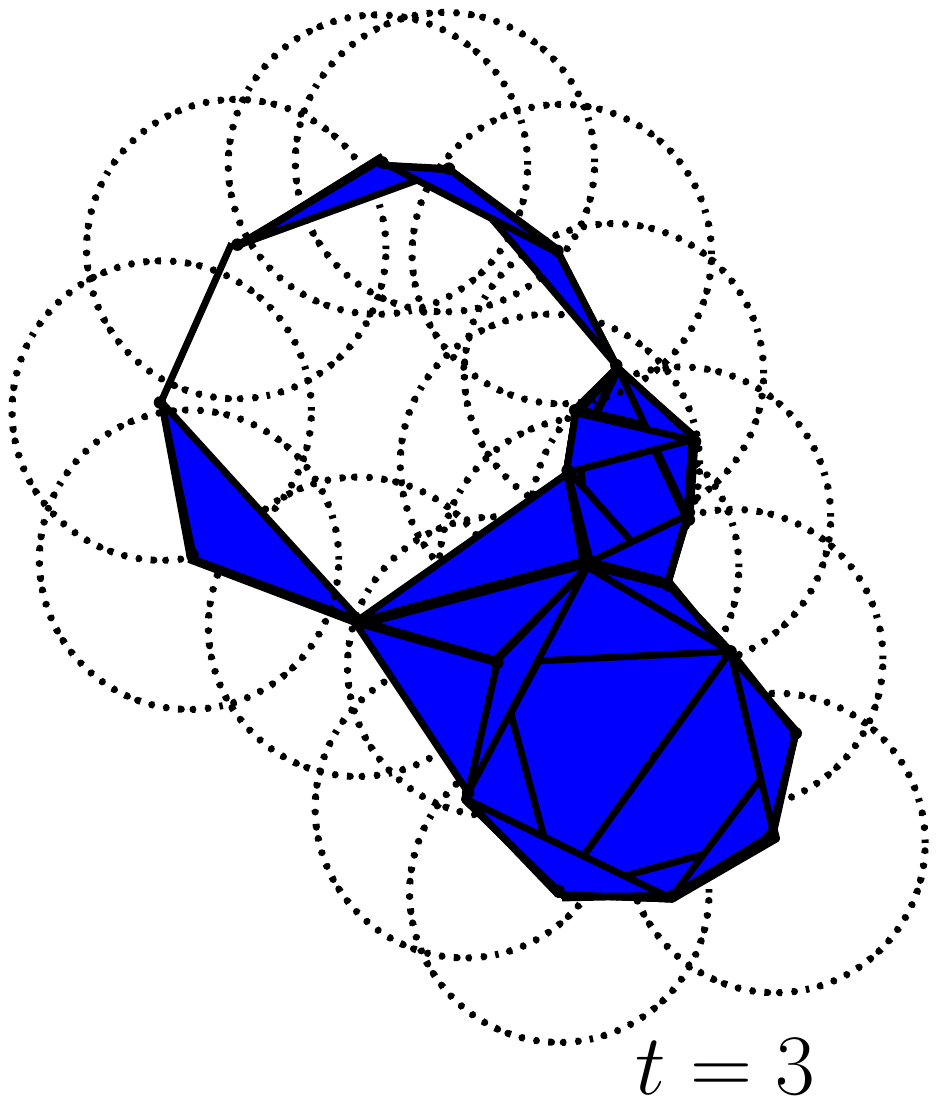}
	\includegraphics[width=0.12\textwidth]{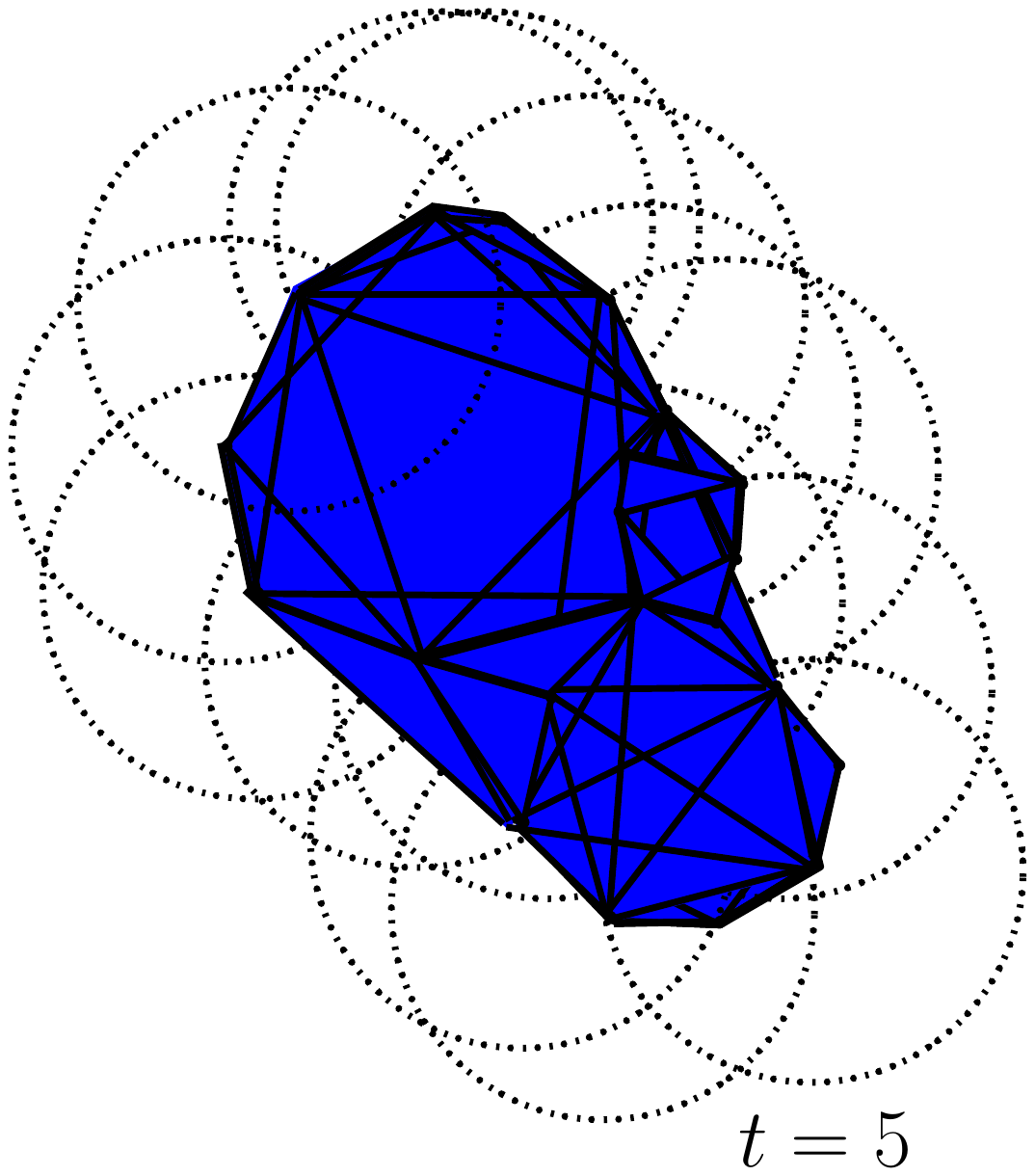}
	\includegraphics[width=0.3\textwidth]{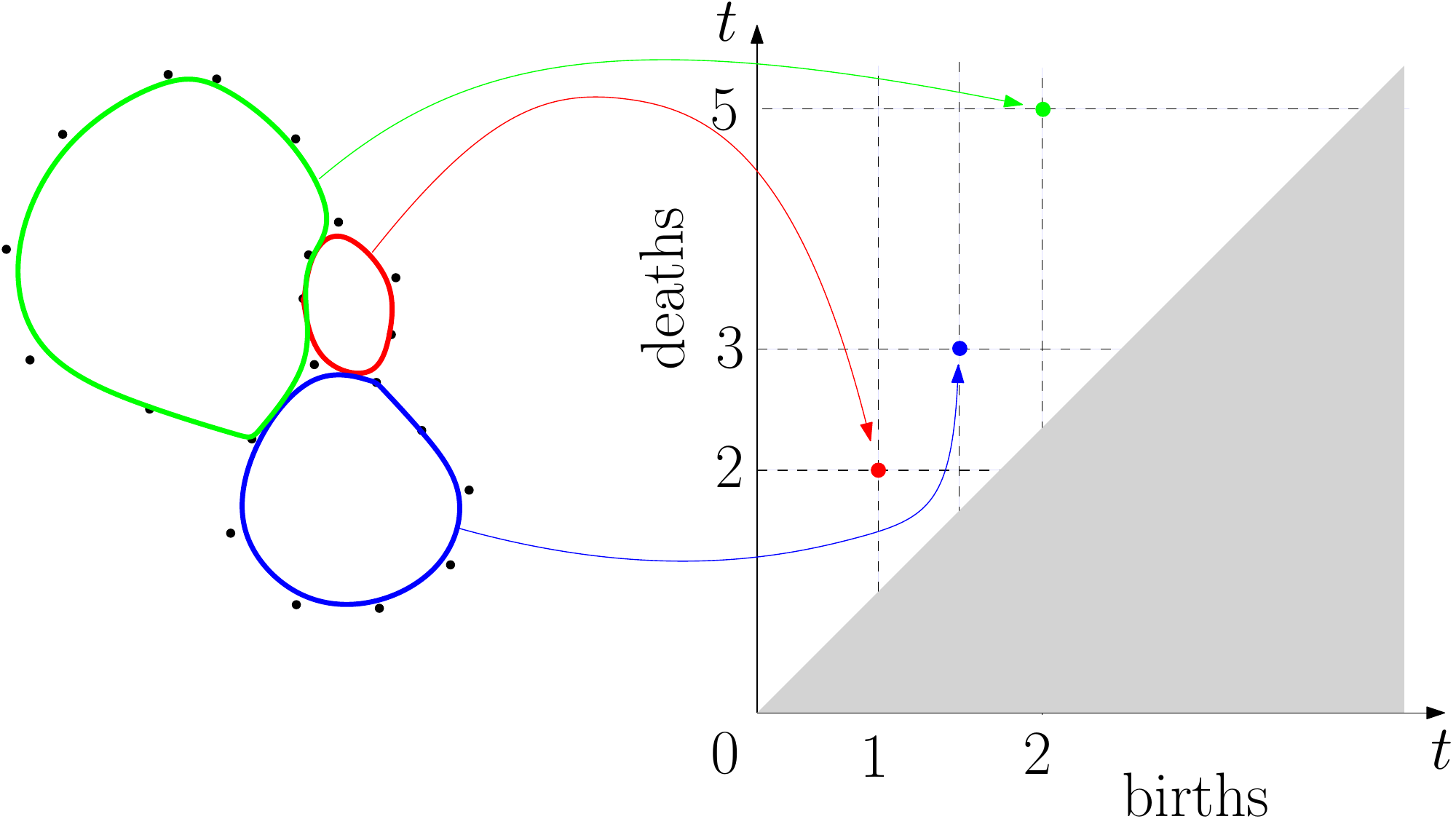}
	\vspace{-.2cm}
	\caption{\Cech{} filtration on a 2D point cloud in dimension $D=1$ (recording loops) and the corresponding PD.}
	\vspace{-.5cm}
	\label{fig:Cech_illu}
\end{figure*}

\textbf{Contributions.} We consider the situation where one has access to a $n$-sample of PDs $\mu_1,\dots,\mu_n$ following some (unknown) law $P$. 
A natural way to estimate the EPD of $P$ is to consider its empirical counterpart, which simply reads $\overline{\mu}_{n} \defeq \frac{1}{n}(\mu_1+\dots + \mu_n)$. 
By leveraging techniques from optimal transport theory, we show in \cref{sec:minimax_EPD} that $\overline{\mu}_n$ approximates $\EPD$ at the parametric rate $n^{-1/2}$ with respect to the loss $\Dp^p$ under non-restrictive assumptions, and that it is optimal from a minimax perspective.  
In practice, the support of the measure $\overline{\mu}_{n}$ is obtained as the union of the support of each diagram and tends to be very large if $n \gg 1$, hindering the use of this empirical descriptor in applications. 
To overcome this issue, we propose in \cref{sec:quantiz} an online algorithm to compute a quantization of the empirical EPD and show that---provided a good initialization---the output of our algorithm approximates a quantization of the EPD at an appropriate rate.
For the sake of conciseness, proofs have been deferred to the supplementary material along with code to reproduce our experiments. 

\textbf{Related Work.} \citet{tda:divol2018density} show that under mild assumptions the EPD is a measure with density supported on the half-plane $\upperdiag$, and propose an estimation procedure of the EPD based on kernel density estimation. 
However, they defined convergence in terms of $L_2$ metrics between densities instead of the more natural diagram metric $\Dp$ considered in this work and did not exhibit rates of convergence. 
In optimal transport literature, the study of convergence rates between a measure and its empirical counterpart for the Wasserstein distance $W_p$ dates back to \citep{dudley1969speed}, while more recent papers \citep{ot:singh2018minimax, fournier2015rate, kloeckner2020empirical, lei2020convergence} provide tight controls of the convergence rate of the quantity $W_p^p$. 
There are however two main differences between this line of results and our framework. First, despite both being optimal transport metrics, there exist key differences between the metric $\Dp$ and the Wasserstein metric $W_p$ (see \cref{sec:background}). 
Furthermore, we are not in the common situation where one observes i.i.d.~realizations $X_1,\dots,X_n$ in $\Omega$ and considers the empirical measure $\frac{1}{n}(\delta_{X_1}+\dots+ \delta_{X_n})$ but in the more general setting where one observes measures $\mu_1, \dots, \mu_n$ on $\Omega$ following some law $P$ and considers the distance between the expected measure $\EPD$ and its empirical counterpart $\frac{1}{n} (\mu_1+ \dots + \mu_n)$.

The problem of quantization of measures, namely approximating a given measure with another measure with support of fixed size, has been studied in depth when those measures are supported on $\R^d$ equipped with its natural Euclidean geometry, see for instance \citep{graf2007foundations,fischer2010quantization,levrard2015nonasymptotic,ot:bourne2018semi}. In the context of PDs, where the quantization problem is generally referred to as computing \emph{codebooks} or \emph{bag-of-words} \citep{tda:zielinski2018persistenceBow,tda:zielinski2020persistenceCodebook}, existing methods propose to quantize PDs running a $k$-mean algorithm on the diagram points. 
The intuition that points in a diagram that are close to the boundary $\thediag$ of the half-plane $\upperdiag$ represent less important topological features is taken into account through the introduction of weight functions, requiring to introduce an important hyper-parameter whose choice is unclear in general. 
Our approach differs from the latter on two aspects: first, we do not quantize a single diagram (should it be a superposition of diagrams as in \citep{tda:zielinski2020persistenceCodebook}) but work in an \emph{online} fashion with a sequence of observed diagrams. Second, we work with the standard diagram metric $\Dp$. 
In doing so, we directly take the boundary $\thediag$ into account in the formulation of our problem without needing to introduce a weight function. 
Our quantization algorithm significantly builds on \citep[Alg.~2]{chazal2020optimal}. 
The main difference is that \citeauthor{chazal2020optimal} intend to quantize a measure with respect to the $2$-Wasserstein distance on $\R^d$, while we work with the metric $\Dp$ on $\upperdiag \subset \R^2$. 
This change of perspective introduces some specificities in our problem and allows us to derive results more suited to the context of persistence diagrams. 
Furthermore, while standard algorithms work with $p=2$, we propose a simple variation to encompass the case $p=+\infty$, central in TDA as one retrieves the so-called \emph{bottleneck} distance. 

\vspace{-.2cm}
\section{Background}
\label{sec:background}
\vspace{-.1cm}
\textbf{Persistence diagrams (PDs).} 
Let $X$ be a topological space and let  $f : X \to \R$ be a real-valued continuous function. The \emph{sublevel sets} of $(X,f)$ are defined as $\FF_t \defeq \{ w \in X,\ f(w) < t\}$. 
As the scale parameter $t$ increases from $-\infty$ to $+\infty$, one observes a nested sequence of sets called the \emph{filtration} of $X$ by $f$. 
 Given a fixed dimension $D$, persistent homology (see \citep{tda:edelsbrunner2010computational} for an introduction) provides tools to record the scales at which a topological feature (a connected component for $D=0$, a loop for $D=1$, a cavity for $D=2$, etc.) appears or disappears in the sublevel sets. For instance, a loop (one-dimensional topological component) might appear at some scale $t_1$ (its birth time) in the sublevel set $\FF_{t_1}$, and disappear (``get filled'') at some scale $t_2 > t_1$. One says that the loop \emph{persists} over the interval $[t_1, t_2]$. This results in a collection of intervals\footnote{In the greatest generality, there may be some intervals of the form $[t_1,+\infty)$. In the following, such intervals are simply discarded if ever present.}---each of them accounting for the presence of a topological feature recorded in the filtration process---that can be encoded as a multiset supported on the open half-plane $\upperdiag = \{ x=(t_1, t_2),\ t_2 > t_1 \} \subset \R^2$, or, equivalently, as a locally finite discrete measure  $\dgm(f) \defeq \sum_{i} \delta_{x_i}$, where $\delta_{x_i}$ denotes the Dirac mass located at $x_i \in \upperdiag$. 
Of particular interest is the case where $X = \R^d$, and $f : w \in \R^d \mapsto \dist(w, A)$ is the distance function to $A$ a compact subset of $\R^d$ (for instance a point cloud), see Figure \ref{fig:Cech_illu}. 
The corresponding diagram, called the \Cech{} persistence diagram of $A$, will be denoted by $\dgm(A)$.

\textbf{Metrics for PDs.} Let $\| \cdot \|$ be the Euclidean norm and let $\supp(\mu)$ denote the support of a measure $\mu$. Let $\thediag \defeq \{(t,t),\ t \in \R \}$ be the diagonal (which is also the boundary of $\upperdiag$), and $\groundspace \defeq \upperdiag \sqcup \thediag$. Given $1 \leq p < \infty$, and two measures $\mu, \nu$ supported on $\upperdiag$, one can define the distance between $\mu$ and $\nu$ using an \emph{optimal partial transport} metric:
\begin{equation}\label{eq:Dp_def}
	\Dp(\mu,\nu) \defeq\hspace{-.1cm} \inf_{\pi \in \Adm(\mu,\nu)} \hspace{-.1cm}\left( \iint_{\groundspace \times \groundspace} \|x - y \|^p \dd \pi \right)^{\frac{1}{p}}\hspace{-.2cm},
\end{equation}
where $\Adm(\mu,\nu)$ is the set of measures supported on $\groundspace \times \groundspace$ whose first (resp.~second) marginal coincides with $\mu$ (resp.~$\nu$) on $\upperdiag$ (note in particular that $\pi$ is not constrained on $\thediag \times \thediag$). 
The definition is extended to $p=\infty$ by replacing $\left( \iint_{\groundspace \times \groundspace} \|x - y \|^p \dd \pi \right)^{\frac{1}{p}}$ by $\sup \{\|x-y\|,\ (x,y)\in \supp(\pi)\}$, and the distance $\mathrm{OT}_\infty$ is called the \emph{bottleneck distance}, central in TDA due to its strong stability properties \citep{tda:cohen2007stability,tda:chazal2016structure}. 
Let $\|x - \thediag\| = (t_2 - t_1)/\sqrt{2}$ be the \emph{persistence} of a point $x=(t_1,t_2)\in\upperdiag$, that is its distance to the diagonal $\thediag$. 
The space $(\MM^p, \Dp)$ of \emph{persistence measures} is defined as the space of (non-negative) Radon measures $\mu$ supported on $\upperdiag$ that have finite \emph{total persistence}, i.e.~$\Pers_p(\mu)\defeq \int \|x-\thediag\|^p \dd \mu(x) < \infty$ (this condition ensures that $\Dp$ is always finite). 
Note that the distance $\Dp$ is not only defined for PDs (elements of $\DD$), but for measures on $\upperdiag$ with arbitrary support, therefore making it possible to define a similarity notion between a PD and a more general measure such as an EPD, a crucial aspect of this work.
\vspace{-.3cm}

The metrics $\Dp$ are similar to the \emph{Wasserstein distances} used in optimal transport \citep[Ch.~5]{otam}: for $\sigma,\tau$ two measures \emph{having the same total mass} on a metric space $(S,\rho)$, the distance $W_{p,\rho}(\sigma,\tau)$ is defined as the infimum of $\left(\int_{S^2}\rho(x,y)^p\dd\pi(x,y)\right)^{1/p}$ over all transport plans $\pi$ between $\sigma$ and $\tau$, i.e.~measures on $S \times S$ which have for first (resp.~second) marginal $\sigma$ (resp.~$\tau$). When $\rho$ is the Euclidean distance we write $W_p$ instead of $W_{p,\rho}$. Despite those similarities, there is however a crucial difference between the Wasserstein distance and the $\Dp$ distance: the constraints in \eqref{eq:Dp_def} only involves the marginals on $\upperdiag$, allowing us to transport mass to and from the boundary of the space $\thediag$. 
It makes, in particular, the distance $\Dp$ between measures of different total masses well-defined.
The metrics $\Dp$ were introduced by \citet{ot:figalli2010newTransportationDistance} as a way to study the heat equation with Dirichlet boundary conditions, but \citet{tda:divol2019understanding} observed that these metrics actually coincide with the standard metrics used to compare persistence diagrams \citep[Ch.~8]{tda:edelsbrunner2010computational}.

\vspace{-.1cm}
\textbf{Expected persistence diagrams.} 
Let $P$ be a probability distribution supported on $(\MM^p, \OT_p)$. Let $\EPD$ be the measure defined by, for $A \subset \upperdiag$ compact, 
\begin{equation}
 	\EPD (A) \defeq \E_P[ \mu(A) ],
 \end{equation}
where $\mu \sim P$, and $\mu(A)$ is the (random) number of points of $\mu$ that belongs to $A$. 
This deterministic measure, called the \emph{expected persistence diagram} (EPD) of $P$, was introduced in \citep{tda:divol2018density} were authors proved that, under mild assumptions, it admits a density with respect to the Lebesgue measure on $\upperdiag$. 
Importantly, the EPD is a persistence measure but not a PD in general.

\vspace{-.2cm}

\section{Minimax estimation of the EPD}\label{sec:minimax_EPD}
\vspace{-.1cm}
Let $P$ be a distribution of PDs, and $\EPD$ be its EPD. Given a $n$-sample $\mu_1,\dots,\mu_n$ of law $P$, the empirical EPD is defined as $\overline{\mu}_n \defeq \frac{1}{n} \sum_i \mu_i$. In this section, we control the distance $\Dp^p(\overline\mu_n,\EPD)$ under moment assumptions on the underlying law $P$. 
Note that, according to \citep[Thm.~ 3.7]{tda:divol2019understanding} and the law of large numbers, $\overline\mu_n \xrightarrow{\Dp} \EPD$ almost surely under the minimal assumption that $\E_{P}[\Pers_p(\mu)]<\infty$ (see \cref{lem:as_convergence} in the supplementary material). 
Our goal here is to understand the rate at which this convergence holds.

Let $A_L$ be the $\ell_1$-ball in $\R^2$ centered at $(-L/\sqrt{8},L/\sqrt{8})$ of radius $L/\sqrt{2}$. For $0\leq q \leq \infty$ and $L,M>0$, we let $\MM^q_{L,M}$ be the set of measures $\mu \in\MM^q$ which are supported on $A_L$, with $\Pers_q(\mu)\leq M$. Let $\PP^q_{L,M}$ be the set of probability distributions which are supported on $\MM^q_{L,M}$. 
It is known that persistence diagrams belong to the set $\MM^q_{L,M}$ under non-restrictive assumptions. Namely, we have the following result.

\begin{lemma}[\citet{tda:cohen2010lipschitzStablePersistence}]\label{lem:cohensteiner}
Let $X$ be a $d$-dimensional compact Riemannian manifold, and let $f:X\to \R$ be a Lipschitz continuous function. Then, for every $q>d$, $\dgm(f)$ belongs to $\MM^q_{L,M}$ for some $L$, $M$ depending on $X$, $q$ and the Lipschitz constant of $f$.
\end{lemma}

In particular, for $q>0$, no constraints on the total number of points of the persistence diagram are imposed. This is particularly interesting in applications, where the number of points in PDs is likely to be large, while their total persistence $\Pers_q$ may be moderate, see e.g.~\citep{tda:divolpolonik} for asymptotics in the case of the \v Cech persistence diagrams of large samples on the cube.

\begin{theorem}\label{thm:minimax_result}
Let $1\leq p<\infty$ and $0\leq q < p$. Let $P\in \PP^q_{L,M}$ and let $\mu_1,\dots,\mu_n$ be a $n$-sample from law $P$. If $\overline\mu_n$ is the associated empirical EPD, then,
\begin{equation}
\E[\Dp^p(\overline{\mu}_{ n},\EPD)] \leq c ML^{p-q} \hspace{-.05cm} \p{\frac{1}{ n^{1/2} }              \hspace{-.05cm}      +\hspace{-.05cm} \frac{a_p(n)}{n^{p-q}}}\hspace{-.05cm},
\end{equation}
where $c$ depends on $p$ and $q$, and $a_p(n)=1$ if $p>1$, $\log(n)$ if $p=1$.
\end{theorem}
In particular, if $p\geq q+1/2$, we obtain a parametric rate of convergence of $n^{-1/2}$. This is always the case if $q=0$, i.e.~if we assume that all the diagrams sampled according to $P$ have less than $M$ points. According to Lemma \ref{lem:cohensteiner}, it is also the case if $\mu_i=\dgm(f_i)$ for some random $1$-Lipschitz functions $f_i:X\to \R$, where $X$ is a $d$-dimensional compact Riemannian manifold with $p>d+1/2$.

From a statistical perspective, it is natural to wonder if better estimates of $\EPD$ exist. 
 A possible way to answer this question is given by the minimax framework. Let $\PP$ be a set of probability distributions on $\MM^p$. The minimax rate for the estimation of $\EPD$ on $\PP$ is
  \vspace{-.15cm}
\begin{equation}
\mathcal{R}_n(\PP)\defeq \inf_{\hat \mu_n}\sup_{P\in \PP} \E[\Dp^p(\hat\mu_{ n},\EPD)], 
 \vspace{-.15cm}
\end{equation}
where the infimum is taken over all possible estimators of $\EPD$. An estimator attaining the rate $\mathcal{R}_n(\PP)$ (up to a constant) is called minimax, i.e.~an estimator is minimax on the class $\PP$ if it has the best possible risk uniformly on this class. We show that the empirical EPD $\overline \mu_n$ is a minimax estimator on $\PP^q_{L,M}$ as long as $p \geq q + 1/2$. The case $p=\infty$ is discussed in  \cref{rem:rip_estimation} (supplementary material).

\begin{theorem}\label{thm:minimax}
 Let $1\leq p <\infty$ and $q\geq 0$, $L,M>0$. One has, for some $c$ depending on $p$ and $q$,
   \vspace{-.1cm}
\begin{equation}
\mathcal{R}_n(\PP^q_{L,M}) \geq c ML^{p-q}n^{-1/2}.
 \vspace{-.15cm}
\end{equation}
\end{theorem}

As the EPD $\EPD$ is known to have a smooth density in a wide variety of settings \citep{tda:divol2018density}, it could be expected (likewise it is the case in density estimation \citep{tsybakov}), that one could make use of this regularity to obtain substantially faster minimax rates on appropriate models. Surprisingly enough, using results from statistical optimal transport theory, we show that whatever regularity is assumed on the EPD, no estimators can perform better than the empirical EPD $\overline\mu_n$ for the $\Dp$ loss (from a minimax perspective). Let $B^s_{p',q'}$ be the set of functions $\Omega\to \R$ in the Besov space of parameters $s\geq 0$ and $1\leq p',q'\leq \infty$, see \citep{hardle2012wavelets} for an introduction to Besov spaces; note that this formalism encompasses all $\CC^k$ classes. Consider the model $\PP^{q,s}_{L,M,T}$ of probability distributions $P\in \PP^q_{L,M}$ whose EPD $\EPD$ belongs to $B^s_{p',q'}$ with associated norm smaller than $T/M$.

\begin{theorem}\label{thm:regularityEstimation}
 Let $1\leq p <\infty$, $q,s\geq 0$, $L,M,T>0$ and $1\leq p',q'\leq \infty$. One has
\begin{equation}
\mathcal{R}_n(\PP^{q,s}_{L,M,T}) \geq c ML^{p-q}n^{-1/2},
\end{equation}
where $c$ depends on $s, p', q', p,q$ and $T$.
\end{theorem}
The proof of \cref{thm:regularityEstimation} is based on a similar result appearing in \citep{ot:weed2019estimation}, where minimax rates of estimation with respect to the Wasserstein distance $W_p$ are given for smooth densities on the cube.

\begin{remark}\label{rem:rip_estimation}
In the usual problem of estimating a measure thanks to a $n$-sample with respect to the Wasserstein distance, it has been noted several times \citep{trillos2015rate, ot:weed2019estimation, divol2021short} that this problem becomes significantly easier if the measure has a lower bounded density on its domain. In particular, it is known that the risk for the $W_p^p$ loss of the empirical measure attains the faster rate $n^{-p/2}$ (instead of $n^{-1/2}$) under this hypothesis. If such a result is likely to hold for the $\Dp^p$ loss under similar hypothesis, requiring that the EPD has a lower bounded density on some bounded domain $U$ in $\Omega$ appears to be unreasonable. Indeed, this would imply that the density exhibits a sharp change of behavior at the boundary of $U$, whereas the density of the EPD is known to be typically smooth on $\Omega$ \citep{tda:divol2018density}. Whether there exists a more realistic assumption on the EPD for which the rate of convergence of the empirical EPD is $n^{-p/2}$ remains an open question.
\end{remark}

\vspace{-.2cm}
\section{Quantization of the EPD}\label{sec:quantiz}
\vspace{-.1cm}
This section consists of two steps. In \cref{subsec:quantiz_generalities}, we introduce and study the problem of quantizing persistence measures with respect to the metric $\Dp$, proving in particular the existence of optimal quantizers in general. \cref{subsec:quantiz_epd_algo} provides an online algorithm specifically designed to quantize EPD based on a sequence of observed diagrams $\mu_1, \dots, \mu_n$ and provide theoretical guarantees of convergence. 
\vspace{-.2cm}

\subsection{Quantization for persistence measures.}
\label{subsec:quantiz_generalities}


Let $\mu \in \MM^p$ be a persistence measure and $k$ be a fixed integer. 
The goal of the quantization problem is to build a measure $\nu = \sum_{j=1}^k m_j \delta_{c_j}$ supported on a set of $k$ points $\bc = (c_1, \dots, c_k)$ called a \emph{codebook} (while the $(c_j)_j$s are called \emph{centroids}) that approximates $\mu$ in an optimal way. 
Existing works (including previous works in the TDA literature) treat this problem over the space of probability measures equipped with the Wasserstein metric $W_p$ over $\R^d$. 
Here, we use the metric $\Dp$ instead, more suited to PDs, leading to benefits discussed in \cref{rem:Dp_is_gucci} below. 
Our problem consists in minimizing the quantity $((m_1, c_1), \dots, (m_k,c_k)) \mapsto \Dp\left(\sum_j m_j \delta_{c_j}, \mu\right)$ where $m_j \in \R_+$ and $c_j \in \upperdiag$. 
However, we show in \cref{lemma:restrictToCodebook} below that---as in the standard problem using the metric $W_p$---this problem can be reduced to an optimization problem on the codebook $\bc \in \upperdiag^k$ only. 
To that aim, we introduce a notion of \emph{Voronoï tesselation} relative to a codebook $\bc$, with the subtlety that points closer to the diagonal $\thediag$ define a specific cell, see \cref{fig:partition} for an illustration.

\begin{definition}Let $\bc = (c_1 \dots c_k) \in \upperdiag^k$ and denote by convention $c_{k+1} \defeq \thediag$, so that in particular $\|x - c_{k+1}\| \defeq \|x - \thediag\|$. 
Define for $1 \leq j \leq k+1$, 
\begin{equation} \label{eq:def_voronoi}
\begin{aligned}
	V_j(\bc) \defeq &\{ x \in \upperdiag,\ \forall j' < j, \|x - c_j\| \leq \| x - c_{j'}\| \\
							& \text{ and } \forall j' > j, \| x - c_j\| < \| x - c_{j'}\|\}, \\
								N(\bc) \defeq &\{ x \in \upperdiag,\ \exists j < j' \text{ such that } x \in V_j(\bc) \\
	&\text{ and } \|x - c_j \| = \| x - c_{j'} \| \}.
\end{aligned}
\vspace{-.2cm}
\end{equation}
\end{definition}
Observe that $V_1(\bc), \dots, V_{k+1}(\bc)$ form a partition of $\upperdiag$.

\begin{figure}
\center
\includegraphics[width=0.8\columnwidth]{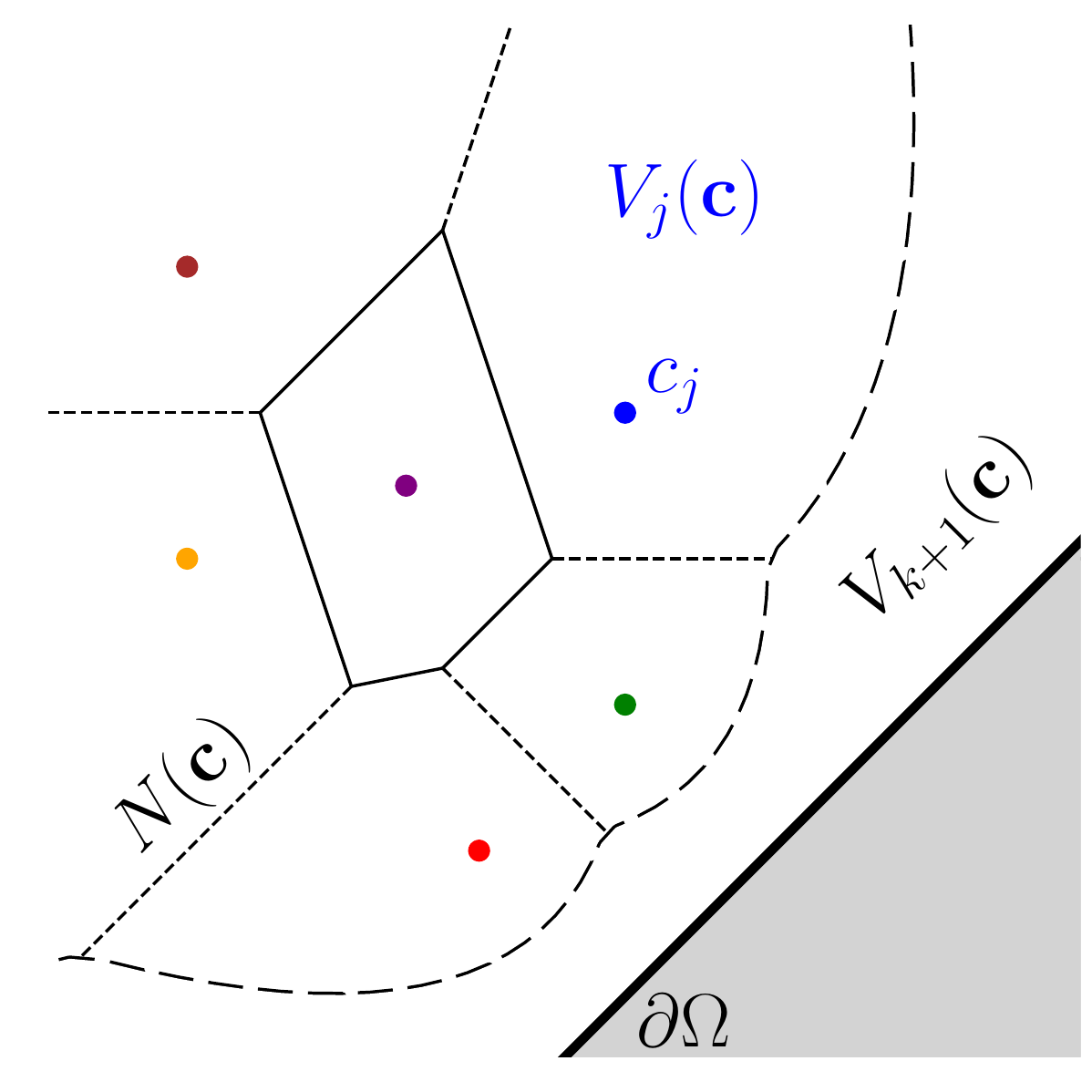}
\vspace{-.5cm}
\caption{Example of partition $V_1(\bc),\dots,V_{k+1}(\bc)$ for a given codebook $\bc$.}
\vspace{-.3cm}
\label{fig:partition}
\end{figure}

\begin{remark}\label{rem:Dp_is_gucci}
    The difference between our approach and previous ones (in particular \citep{chazal2020optimal}) lies in the presence of the ``diagonal cell'' $V_{k+1}(\bc)$. This cell introduces parabolic-shaped boundaries which slightly change the geometry of our problem. However, it has two major benefits. First, it enables a natural geometric identification of points close to the diagonal (which play a specific role in TDA) through the cell $V_{k+1}$ and we do not ``waste'' centroids $(c_j)_{j=1}^k$ to encode them. Second, our approach does not require the introduction of a weight function (that artificially lowers the mass of points close to the diagonal), as typically done; removing the dependency on an important hyper-parameter.
\end{remark}

The following lemma states that given a persistence measure $\mu$ and a codebook $\bc = (c_1,\dots,c_k)$, it is always optimal to set $m_j = \mu(V_j(\bc))$. 
\begin{lemma}\label{lemma:restrictToCodebook}
	Let $\bc = (c_1, \dots, c_k)$. Let $\hat{\mu}(\bc) \defeq \sum_{j=1}^{k} \mu(V_j(\bc)) \delta_{c_j}$. Let $\nu = \sum_{j=1}^k m_j \delta_{c_j}$ for some $m_1,\dots,m_k \geq 0$. Then $\Dp(\hat{\mu}(\bc) , \mu) \leq \Dp(\nu, \mu)$.
\end{lemma}
Therefore, quantizing $\mu$ boils down to the choice of the codebook $\bc$. Formally, given a persistence measure $\mu$ to be quantized, a parameter $1 \leq p < \infty$ and an integer $k$, the quantization problem in the space of persistence measures consists in minimizing $R_{k,p}: \upperdiag^k \to \R$ defined for $\bc\in \upperdiag^k$ by
\begin{equation}\label{eq:quantiz_in_PD_space}
\begin{aligned}
R_{k,p}(\bc)&\defeq\Dp(\hat{\mu}(\bc), \mu) \\
		&= \left(\sum_{j=1}^{k+1} \int_{V_j(\bc)} \|x - c_j\|^p \dd \mu(x)\right)^{\frac{1}{p}},
	\end{aligned}
\end{equation}
To alleviate notations, we write $R_k$ instead of $R_{k,p}$ when the parameter $p$ does not play a significant role. 
The value $R_k(\bc)$ is called the \emph{distortion} achieved by $\bc$. 
Let $R_k^* \defeq \inf_{\bc \in \upperdiag^k} R_k(\bc)$ and let $\bC_k \defeq \argmin_{\bc\in \upperdiag^k} R_k(\bc)$ be the set of optimal codebooks. 
Note that $R_k^* = 0$ if (and only if) $|\supp(\mu)| \leq k$. 
From now on, \emph{we assume that $\mu$ has at least $k$ points in its support.}

We can now state the main result of this subsection: the existence of an optimal codebook $\bc^*$ for any persistence measure in $\MM^p$. This result shares key ideas with \citep[Thm~4.12]{graf2007foundations}, although we replace the assumption of finite $p$-th moment of the measure to be quantized by the assumption of finite total persistence $\Pers_p(\mu) < \infty$, more natural in TDA ($\mu$ may even have infinite total mass in our setting).

\begin{prop}[Existence of minimizers]\label{prop:existence_and_prop_of_minimizers}
The set of optimal codebooks $\bC_k$ is a non-empty compact set. Furthermore, if $\bc^*\in\bC_k$, then, for all $1\leq j \neq j'\leq k$, $\mu(V_j(\bc^*))>0$ and $c_j^* \neq c_{j'}^*$.
\end{prop}

\begin{corollary}\label{coro:positive_quantities}
The following quantities are positive:
	\begin{equation}
		\begin{aligned}
			\Dmin &\defeq \inf_{\bc^* \in \bC_k, 1 \leq j \neq j' \leq k+1} \| c^*_j - c^*_{j'}\|, \\
			\mmin &\defeq \inf_{\bc^* \in \bC_k, 1 \leq j \leq k} \mu(V_j(\bc^*)). \\
		\end{aligned}
	\end{equation}
\end{corollary}
 \vspace{-.2cm}

\textbf{Computational aspects.} One could consider to numerically solve the quantization problem \eqref{eq:quantiz_in_PD_space} deriving optimization algorithms based on their counterpart in the optimal transport literature \citep{ot:cuturi2014fast}, see \citep[\S 7.2]{tda:lacombe:tel-02979251} for instance. 
However, using such techniques to quantize empirical EPDs would not be satisfactory for two reasons. 
First, the empirical EPD has in general a large number of points, hindering computational efficiency. 
Second, we want to leverage the fact that we observe a sequence of diagrams $\mu_1, \dots, \mu_n$, and not only their sum, to design an online algorithm that remains tractable with large sequences of large diagrams.

\setlength{\textfloatsep}{.2cm}
\begin{algorithm}[tb]
    \caption{Online quantization of EPDs} 
    \label{alg:online_quantiz}
\begin{algorithmic}
\STATE {\bfseries Input:} A sequence $\mu_1, \dots, \mu_n$, integer $k$, parameter $p$.
\STATE {\bfseries Preprocess:} Divide indices $\{1, \dots, n\}$ into batches $(B_1, \dots, B_T)$ of size $(n_1, \dots, n_T)$. Furthermore, divide $(B_t)_t$ into two halves $B_t^{(1)}$ and $B_t^{(2)}$.
\STATE Set $\overline{\mu}_{t}^{(\alpha)} \defeq \frac{2}{n_t}\sum_{i \in B_t^{(\alpha)}} \mu_i$ for $1 \leq t \leq T, \alpha \in \{1,2\}$.
\STATE {\bfseries Init:} Sample $c^{(0)}_1 \dots c^{(0)}_k$ from the diagrams.
 
 \FOR{$t=0, \dots, T-1$}
   \STATE $\bc^{(t+1)} =  U_p(t, \bc^{(t)}, \overline{\mu}_{t+1}^{(1)}, \overline{\mu}_{t+1}^{(2)}) $ using \eqref{eq:update_centroid}
 \ENDFOR
 \STATE {\bfseries Output:} The final codebook $\bc^{(T)}$.
 \end{algorithmic}
\end{algorithm}

\begin{figure*}
    \includegraphics[width=0.68\textwidth]{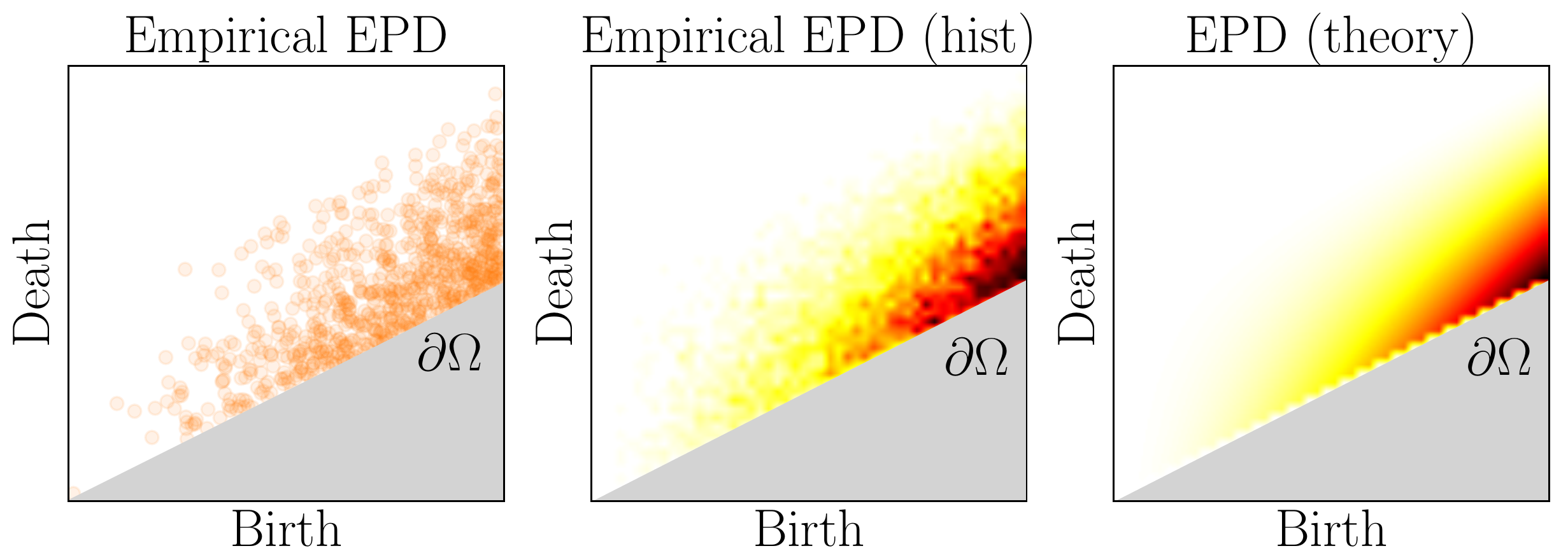}
    \includegraphics[width=0.3\textwidth]{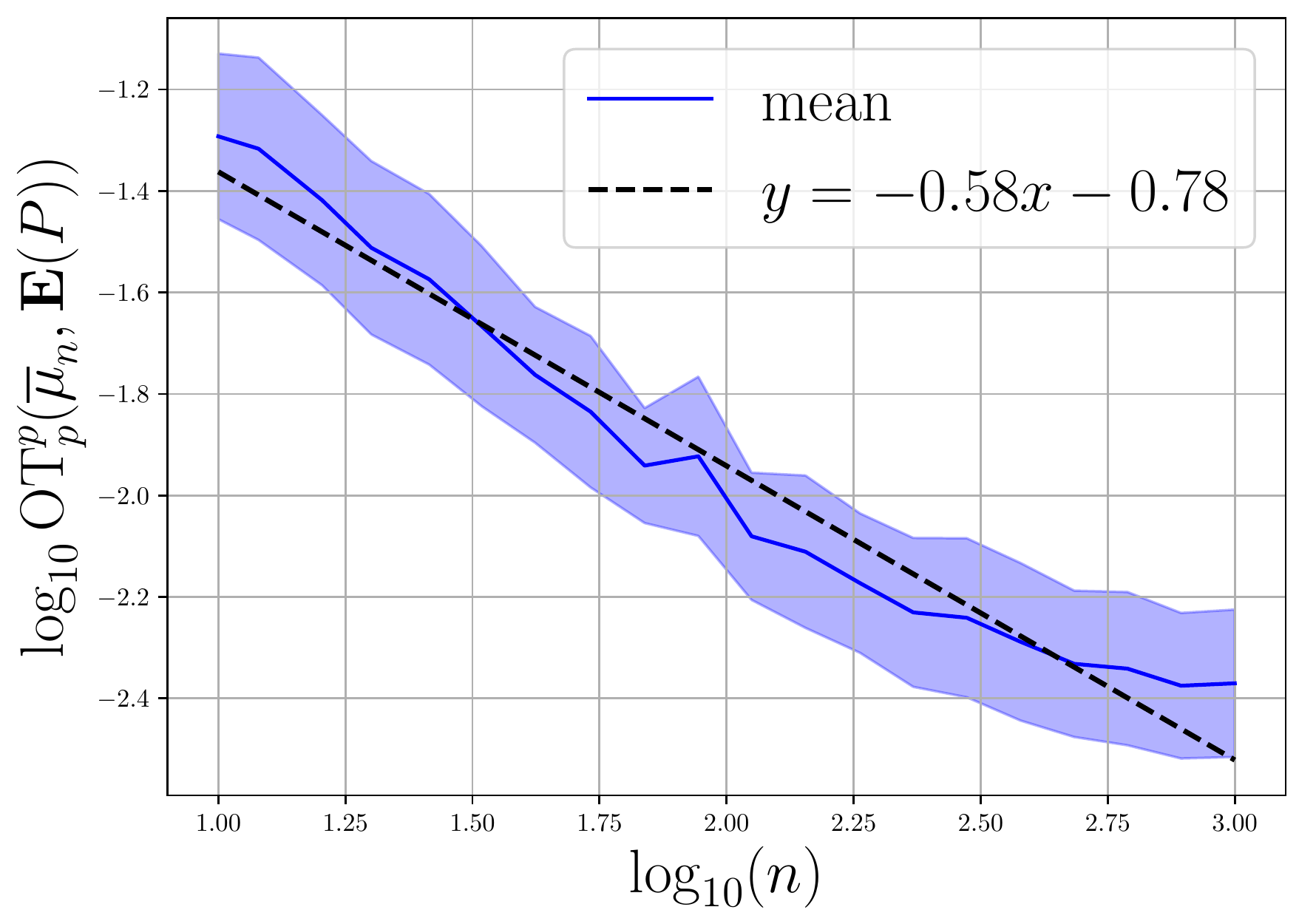} 
           \vspace{-.3cm}
    \caption{From left to right. (a) Empirical EPD $\overline{\mu}_n$ with $n=10^3$. (b) Histogram of the empirical EPD on a $50\times 50$ grid. (c) EPD $\EPD$ of $P$, displayed on the same grid. (d) Distance $\Dp^p(\overline\mu_n,\EPD)$ for $p=2$ for different values of $n$ in log-log scale (mean and standard deviation over 100 runs). A linear regression shows a convergence rate of order $n^{-0.58}$, close to the theoretical rate of $n^{-1/2}$ indicated by \cref{thm:minimax_result}.}
    \vspace{-.5cm}
    \label{fig:epd_convergence}
\end{figure*}

\vspace{-.2cm}
\subsection{Quantization of an empirical EPD}
\label{subsec:quantiz_epd_algo}
In \cref{alg:online_quantiz}, we propose an online algorithm---adapted from \citep[Alg.~2]{chazal2020optimal} to the context of PDs and with arbitrary $p>1$ instead of $p=2$---that takes a sequence of observed PDs $\mu_1, \dots, \mu_n$ (a $n$-sample of law $P$) and outputs a codebook $(c_1,\dots,c_k)$ aiming at approximating $\EPD$. 
The algorithm relies on an update function $U_p$ for $p > 1$ defined as 
\vspace{-.2cm}
\begin{equation} \label{eq:update_centroid}
   U_p(t, \bc, \mu, \mu')\hspace{-.04cm}  \defeq\hspace{-.04cm} \bc - \frac{ \left( \hspace{-.02cm}\frac{\mu(V_j(\bc))}{\mu'(V_j(\bc)  )} \left(c_j  - v_p(\bc, \mu)_j \right)\hspace{-.02cm} \right)_j}{t+1},\hspace{-.07cm}
\end{equation}
where $v_p(\bc,\mu)_j$ is the \emph{$p$-center} of mass of $\mu$ over the cell $V_j(\bc)$:
\vspace{-.4cm}
\begin{equation}\label{eq:p-center-of-mass}
    v_p(\bc, \mu)_j \defeq \argmin_y \left(\int_{V_j(\bc)} \|y - x\|^p \dd \mu(x) \right)^\frac{1}{p}.
    \vspace{-.2cm}
\end{equation}
When $p=2$, one simply has $v_2(\bc,\mu)_j = \int_{V_j(\bc)} x \frac{\dd \mu(x)}{\mu(V_j(\bc))}$ and if in addition $\mu = \mu'$, the update \eqref{eq:update_centroid} simplifies to 
\vspace{-.1cm}
\[ c_j \mapsto \frac{t}{t+1} c_j + \frac{1}{t+1} \int_{V_j(\bc)} x \frac{\dd\mu(x)}{\mu(V_j(\bc))},
\]
so that roughly speaking, we are pushing $c_j$ toward the usual center of mass of $\mu$ over the cell $V_j(\bc)$, similar to what is done when using the Lloyd algorithm to solve the $k$-means problem \citep{lloyd1982least}. 
More generally, \eqref{eq:update_centroid} can be understood as pushing $c_j$ toward the point that would decrease the distortion $R_{k,p}$ over the cell $V_j(\bc)$ the most, using a step-size (or learning rate) $\frac{1}{t+1}$.
There is no closed-form for $v_p$ for $p \neq 2$, though standard convex solvers may be used \citep{gonin1989nonlinear}. 
When $p=+\infty$, a central situation in TDA as it means working with the bottleneck distance $\mathrm{OT}_\infty$, computing $v_\infty$ boils down to get the center of the \emph{smallest enclosing circle} of $V_j(\bc) \cap \supp(\mu)$. When $\mu$ is a discrete measure (e.g.~an empirical EPD), this problem can be solved in linear time with respect to the number of points of $\mu$ that belong to $V_j(\bc)$ \citep{megiddo1983linear}.

Note that in \cref{alg:online_quantiz}, the split of batches $B_t = (B_t^{(1)}, B_t^{(2)})$ is only required for technical considerations (see the supplementary material and \citep{chazal2020optimal}). 
In practice, this algorithm can be used without further assumptions and empirically, using $B_t = B_t^{(1)} = B_t^{(2)}$ yields substantially similar results. 
We provide a theoretical analysis of \cref{alg:online_quantiz} in the case $p=2$, in particular through \cref{thm:online_quantiz} which states that this algorithm is nearly optimal as a way to quantize $\EPD$, provided the initialization is good enough. 
As in \cref{sec:minimax_EPD}, we  consider a probability distribution $P\in \PP^p_{L,M}$. For $t>0$ and $A\subset \upperdiag$, we let $A^t\defeq \{x\in \upperdiag,\ \exists a \in A, \|x-a\| \leq t\}$ be the $t$-neighborhood of $A$.
\begin{definition}[Margin condition]\label{def:margin}
	Let $\bc^*$ be an optimal quantizer of $\EPD$. We say that $P$ satisfies a \emph{margin condition} of parameter $\lambda > 0$ and radius $r_0$ at $\bc^*$ if, for all $t \in [0, r_0]$, one has $\EPD (N(\bc^*)^t) \leq \lambda t.$
\end{definition}
Margin-like conditions on optimal codebook are standard in quantization literature \citep{tang2016lloyd,levrard2018whenKmeansWork}. 
Informally, it indicates that the EPD concentrates around $k$ poles, aside from the mass that is distributed close to the diagonal $\thediag$; the smaller the $\lambda$, the more concentrated the measure. 
Note that this condition holds as long as the $\EPD$ has a bounded density (although with possibly large $\lambda$), a property which is satisfied in a large number of situations, see \citep{tda:divol2018density}. 

\vspace{-.1cm}
The following theorem states that given a $n$-sample of law $P$, \cref{alg:online_quantiz} outputs in $T = \frac{n}{\log(n)}$ steps a codebook $\bc^{(T)}$ that approximates (in expectation) an optimal codebook $\bc^*$ for $\EPD$ at rate $\frac{\log(n)}{n}$, to be compared with the optimal rate of $\frac{1}{n}$ \citep[Prop.~7]{levrard2018whenKmeansWork}. 
It echoes \citep[Thm.~5]{chazal2020optimal} with the difference that, thanks to the diagonal cell $V_{k+1}$, we require a uniform bound on the total persistence of the measures rather than a uniform bound on their total mass, a more natural assumption in TDA. 

\begin{theorem}\label{thm:online_quantiz}
	Let $p=2$. Let $P\in\PP^2_{L,M}$ and let $\bc^*$ be an optimal codebook for $\EPD$. Assume that $P$ satisfies a margin condition at $\bc^*$ with parameters $r_0$ large enough and $\lambda$ small enough (with respect to $\Dmin,\mmin,L$ and $M$).  Let $\mu_1, \dots, \mu_n$ be a $n$-sample of law $P$ and $B_1, \dots, B_T$ be equally sized batches of length $C_1 \log(n)$. Finally, let $\bc^{(T)}$ denote the output of \cref{alg:online_quantiz}. There exists $R_0>0$ such that if $\|\bc^{(0)} - \bc^*\| \leq R_0$, then
\[ \E\| \bc^{(T)} - \bc^* \|^2 \leq  C_2 (\log n)/n,
\]
where $C_1, C_2$ and $R_0$ are constants depending on $p,L, M, k, \Dmin$ and  $\mmin$.
\end{theorem}

\begin{figure*}
    \includegraphics[width=0.9\textwidth]{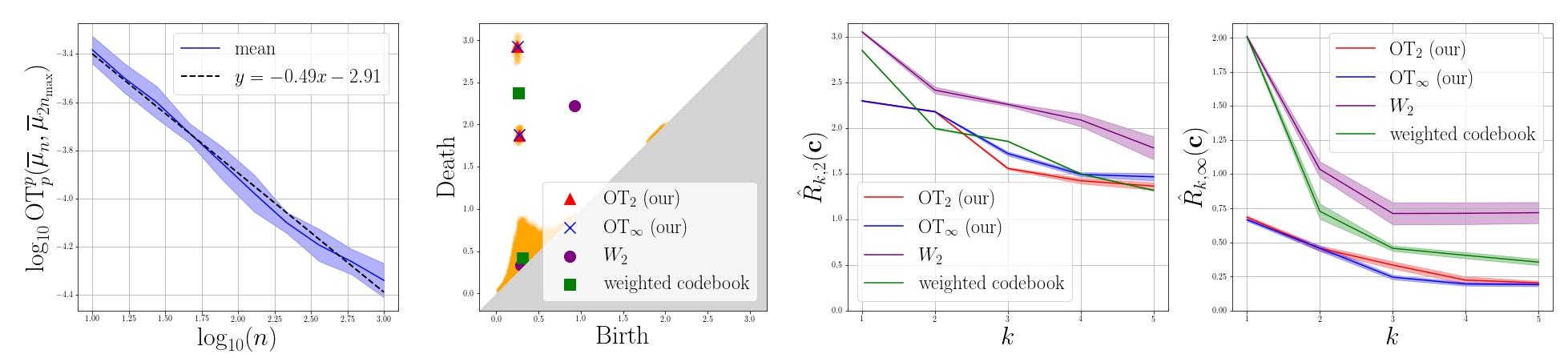}
         \vspace{-.5cm}
    \caption{From left to right: (a) The convergence rate for a point cloud sampled on the surface of a torus, exhibiting a rate of $n^{-1/2}$. (b) The quantization output for the different approaches considered with $k=2$. As our approach accounts for the diagonal through the cell $V_{k+1}$, our codebooks retrieve the two clusters present in the EPD, while other approaches have one centroid used to account for the mass close to the diagonal. (c,d) The average distortion $R_{k,p}$ over $10$ runs for the different methods, with $p=2$ and $p=+\infty$.}
    \vspace{-.4cm}
    \label{fig:torus_expe}
\end{figure*}
\vspace{-.4cm}

\section{Numerical illustrations}
\vspace{-.1cm}
\label{sec:expe}
We now provide some numerical illustrations that showcase our different theoretical results and their use in practice. Throughout, PDs are computed using the \texttt{Gudhi} library \citep{gudhi} and $\Dp$ distances are computed building on tools available from the \texttt{POT} library \citep{flamary2021pot}. See the supplementary material for further implementation details and complementary experiments.

\textbf{Convergence rates for the empirical EPD.}
We first showcase the rate of convergence of Theorem \ref{thm:minimax_result}. There are only few cases where explicit expressions for the EPD of a process are known. For instance, for \v Cech PDs based on a random sample of points, the corresponding EPD is known in closed-form only if the sample is supported on $\R$ \citep[Rem.~4.5]{tda:divolpolonik}. We therefore first consider a simple setting where an explicit expression can be derived. Let $X$ be a set of $N$ triangles $T_1,\dots,T_N$, where $N$ is uniform on $\{1,\dots,20\}$. We let $f:X\to \R$ be a random piecewise constant function, which is equal to $U_{i,j}$ on the $j$th edge of the triangle $T_i$, where the variables $(U_{i,j})$ are i.i.d.~uniform variables on $[0,1]$. 
Furthermore, the function $f$ is equal to $\max_{j=1,2,3} U_{i,j} + V_i$ on the inside of the triangle $T_i$, where the $V_i$s are independent, independent from the $U_{i,j}$s, and follow a Beta distribution $\beta(1,3)$. 
Let $P$ be the distribution of the associated random PD. 
Let $\mathbf{rec}$ be the rectangle $[r_1,r_2]\times[s_1,s_2]$ for $r_1\leq r_2\leq s_1\leq s_2$. Then,
\vspace{-.1cm}
\begin{equation}
     \EPD(\mathbf{rec}) = 30\int_{r_1}^{r_2} t^2\mathbb{P}(s_1-t\leq V\leq s_2-t)\dd t,
     \vspace{-.1cm}
\end{equation}
where $V\sim \beta(1,3)$. In practice, we compute $\EPD$ on a discretization of $[0,1]\times[0,2]$ through a grid of size $50 \times 50$. Meanwhile, we sample empirical EPDs $\overline \mu_n$ for $10 \leq n \leq 10^3$. 
In order to estimate $\Dp^p(\overline{\mu}_n, \EPD)$, we also turn these EPDs into histograms on the same grid, and then compute the $\Dp$ distance between two histograms. See \cref{fig:epd_convergence} for an illustration which showcases in particular the expected rate $n^{-1/2}$.

We also exhibit the convergence of the empirical EPD in a more usual setting for the TDA practitioner. Namely, we build a random point cloud $\mathbf{X}$ with $10^3$ points sampled on the surface of a torus with outer radius $r_1=5$ and inner radius $r_2=2$, and then consider the corresponding random \Cech{} diagram for the $1$-dimensional homology (loops, see \cref{sec:background}). 
Given $n$ realizations of $\mathbf{X}$, we compute the empirical EPD $\overline\mu_n$, where $n$ ranges from $10$ to  $n_{max}=1000$. 
As no closed-form for the corresponding EPD is known, we use as a proxy the empirical EPD based on a sample of size $2n_{\mathrm{max}}$, and then showcase in  \cref{fig:torus_expe} (left) the convergence of $\Dp^p(\overline{\mu}_n, \overline{\mu}_{2n_{\mathrm{max}}})$ at rate $n^{-1/2}$.

\textbf{Quantization of the EPD.} We now illustrate the behavior of \cref{alg:online_quantiz} using $p=2$ and $p=\infty$ (referred to as ``$\mathrm{OT}_2$'' and ``$\mathrm{OT}_\infty$'', respectively) and compare it to two natural alternatives. \citep[Alg.~2]{chazal2020optimal} is essentially the same algorithm without the ``diagonal cell'' $V_{k+1}(\bc)$; as such, centroids are dramatically influenced by points close to the diagonal which are likely to be abundant in standard applications of TDA. It is referred to as ``$W_2$'' in our illustrations, as it relies on quantization with respect to the Wasserstein distance with $p=2$. 
The second alternative, referred to as ``weighted codebook'', is the one proposed in \citep{tda:zielinski2020persistenceCodebook}, which can be summarized in the following way: consider the empirical EPD $\overline{\mu}_n$ built on top of observations $\mu_1,\dots,\mu_n$ (that is, concatenate the diagrams), and then subsample $N$ points in the support of the empirical EPD, with the subtlety that the probability of choosing a point $x \in \supp(\overline{\mu}_n)$ depends on a weight function $w : \upperdiag \to \R_+$. 
Typical choices for $w$ are of the form $w(x) = \min \left(\max\left(0, \frac{(\|x-\thediag\|^q - \lambda)}{\theta - \lambda}\right) , 1 \right)$ for some parameters $(\lambda, q, \theta)$; the goal being to favor sampling points far from the diagonal. 
\citeauthor{tda:zielinski2020persistenceCodebook} propose, in practice, to sample $N=10^4$ points and to set $q=1$, while $\lambda$ and $\theta$ are the $0.05$ and $0.95$ quantiles of the distribution of $\{\|x-\thediag\|^q,\ x \in \supp(\overline{\mu}_n)\}$, respectively. We use these parameters in our experiments.
One then runs the Lloyd algorithm ($k$-means) on the set of $N$ points that have been sampled to obtain a quantization of the empirical EPD. 

 \vspace{-.1cm}
We compare the different approaches in the following experiment. We randomly sample a point cloud $\mathbf{X}$ of size $\mathbf{m}$ on the surface of a torus with radii $(\mathbf{r_1}, \mathbf{r_2})$, where $\mathbf{m}, \mathbf{r_1}, \mathbf{r_2}$ are random variables that respectively follow a Poisson distribution of parameter $m \in \mathbb{N}$, a uniform distribution over $[r_1 - \epsilon, r_1 + \epsilon]$ and a uniform distribution over $[r_2 - \epsilon, r_2 + \epsilon]$. We use $m=2,000, \epsilon=0.1, r_1 = 5$ and $r_2=2$ in our experiments.
Given such a random point cloud $\mathbf{X}$, we build the \Cech{} persistence diagram of its $1$-dimensional features, denoted by $\mu$, leading to a distribution $P$ of PDs. 
We then build a $n$-sample $\mu_1, \dots, \mu_n$ with $n=100$ and, for $k \in \{1, \dots, 5\}$, compute the different codebooks returned by the aforementioned methods, using batches of size $10$ for $\mathrm{OT}_2, \mathrm{OT}_\infty$ and $W_2$. All algorithms are initialized in the same way: we select the $k$ points of highest persistence in the first diagram $\mu_1$. To compare the quality of these codebooks, we evaluate their distortion \eqref{eq:quantiz_in_PD_space} with $p=2$ and $p=\infty$. As we do not have access to the true EPD $\EPD$, we approximate this quantity through its empirical counterpart $\hat{R}_{k,p}(\bc) \defeq \left( \int_\upperdiag \min_{1 \leq j\leq c_{k+1}} \|x - c_j\|^p \dd \overline{\mu}_n(x) \right)^\frac{1}{p}$, with $\hat{R}_{k,\infty}(\bc) = \max_{x \in \supp(\overline{\mu}_n)} \min_j \|x-c_j\|$. 
Results are given in \cref{fig:torus_expe}. Interestingly, when $p=2$ our approach is on a par with the weighted codebook approach, but becomes substantially better when evaluated with $p=\infty$, that is using the bottleneck distance which is the most natural metric to handle PDs.

 \vspace{-.2cm}

\section{Conclusion}
 \vspace{-.1cm}
This work is dedicated to the estimation of expected persistence diagrams, for which we prove that they are approximated, for the natural diagram metrics $\Dp$, by their empirical counterpart in an optimal way from a minimax perspective. 
We then introduce and study the quantization problem in the space of persistence diagrams, proving results of independent interest. Finally, we introduce an online algorithm to estimate a quantization of the EPD with theoretical guarantees. Interestingly, our algorithm can handle the case $p=\infty$, central in TDA, and has the advantage of not requiring hyper-parameters to account for the peculiar role played by the diagonal. 
We illustrate our results in numerical experiments and our code will be made publicly available. We believe that this work offers new perspectives to handle sample of PDs in practice and that it strengthens our understanding of statistical properties of PDs in random settings.

\clearpage
\bibliography{biblio}
\bibliographystyle{icml2021}

\clearpage
\appendix
\section*{Supplementary Material for: Estimation and Quantization of Expected Persistence Diagrams}
\section{Proofs of \cref{sec:minimax_EPD}}

We let $\mu(f)$ denote the integral of some function $f:\upperdiag\to \R$ against the measure $\mu$.
\begin{lemma}\label{lem:as_convergence}
    Let $P$ be a probability measure on $\MM^p$ such that $\E_P[\Pers_p(\mu)]<\infty$. Let $(\mu_n)_{n\geq 1}$ be a sequence of i.i.d.~variables of law $P$ and let $\overline\mu_n=\frac{1}{n}(\mu_1+\cdots+\mu_n)$. Then,
    \begin{equation}
        \Dp(\overline\mu_n,\EPD)\xrightarrow[n\to\infty]{}  0 \text{ almost surely.}
    \end{equation}
\end{lemma}

\begin{proof}[Proof of \cref{lem:as_convergence}]
 By the strong law of large numbers applied to the function $\|\cdot-\thediag\|^p$, we have $\Pers_p(\overline\mu_n)\to \Pers_p(\EPD)$ almost surely. Also, for any continuous function $f:\Omega\to \R$ with compact support, we have $\overline\mu_n(f)\to\EPD(f)$ almost surely. This convergence also holds almost surely for any countable family $(f_i)_i$ of functions. Applying this result to a countable convergence-determining class for the vague convergence, we obtain that $(\overline\mu_n)_n$ converges vaguely towards $\EPD$ almost surely. We conclude thanks to \citep[Thm~3.7]{tda:divol2019understanding}.
\end{proof}

Before proving Theorem \ref{thm:minimax_result}, we give a general upper bound on the distance $\Dp$ between two measures in $\MM^p$. The bound is based on a classical multiscale approach to control a transportation distance between two measures, appearing for instance in \citep{ot:singh2018minimax}. Let  $J\in \N$. For $k\geq 0$, let $B_k = \{x\in  A_L,\ \|x - \thediag\| \in (L2^{-(k+1)},L2^{-k}]\}$. The sets $\{B_k\}_{k\geq 0}$ form a partition of $A_L$.  We then consider a sequence of nested partitions $\{\SS_{k,j}\}_{j=1}^{J}$ of $B_k$, where $\SS_{k,j}$ is made of $N_{k,j}$ squares of side length $\eps_{k,j}=L2^{-(k+1)}2^{-j}$. See also Figure \ref{fig:partition_box}.  Let $\mu_{|B_k}$ be the measure $\mu$ restricted to $B_k$ and $\mu_k = \frac{\mu_{|B_k}}{\mu(B_k)}$ be the conditional probability on $B_k$. If $\mu(B_k)=0$, we let $\mu_k$ be any fixed measure, for instance the uniform distribution on $B_k$.

\begin{lemma}\label{lem:deterministic_control}
Let $\mu,\nu$ be two measures in $\MM^p$, supported on $A_L$. Then, for any $J\geq 0$, with $c_p=2^{-p/2}(1+1/(2^p-1))$,
\[
\begin{split}
    &\Dp^p(\mu,\nu) \leq 2^{p/2}L^p \sum_{k\geq 0} 2^{-kp} \Big( 2^{-Jp} (\mu(B_k)\wedge \nu(B_k))\\
    & +c_p |\mu(B_k)-\nu(B_k)|+ \sum_{\substack{1\leq j\leq J \\ S\in \SS_{k,j-1}}}2^{-jp}|\mu(S)-\nu(S)| \Big).
\end{split}
\]
\end{lemma}

\begin{proof}
Denote by $m_k$ the quantity $\mu(B_k)\wedge\nu(B_k)$.
Let $\pi_k \in \Pi(\mu_k,\nu_k)$ be an optimal plan (in the sense of $W_p$) between the probability measures $\mu_k$ and $\nu_k$. 
If $\mu(B_k)\leq \nu(B_k)$, then $\mu(B_k)\pi_k$ transports mass between $\mu_{|B_k}$ and $\frac{\mu(B_k)}{\nu(B_k)} \nu_{|B_k}$. 
We then build an admissible plan between $\frac{\mu(B_k)}{\nu(B_k)} \nu_{|B_k}$ and $\nu_{|B_k}$ by transporting $\p{1-\frac{\mu(B_k)}{\nu(B_k)}}\nu_{|B_k}$ to the diagonal, with cost bounded by $\p{1-\frac{\mu(B_k)}{\nu(B_k)}}\nu(B_k)(L2^{-k})^p$. 
Acting in a similar way if $\nu(B_k)\leq \mu(B_k)$, we can upper bound $\Dp^p(\mu,\nu)$ by
     \begin{equation}\label{eq:cost_transport1}
\sum_{k\geq 0} \p{m_kW_p^p(\mu_k,\nu_k) + L^p2^{-kp}|\mu(B_k)-\nu(B_k)|}.
\end{equation}
Lemma 6 in \citep{ot:singh2018minimax} shows that 
\begin{equation}\label{eq:lemma_from_singh}
\begin{split}
&W_p^p(\mu_k,\nu_k)\leq 2^{p/2}L^p2^{-(k+1)p}\Big(2^{-Jp} \\
&\qquad+ \sum_{\substack{1\leq j\leq J \\ S\in \SS_{k,j-1}}}2^{-jp}  |\mu_k(S)-\nu_k(S)|\Big).
\end{split}
\end{equation}
Furthermore, one can check that for any $S \subset B_k$
\begin{align*}
&m_k|\mu_k(S)-\nu_k(S)| \leq\\
& |\mu(S) - \nu(S)| + \frac{\nu(S)\wedge \mu(S)}{\mu(B_k)\vee \nu(B_k)} |\mu(B_k)-\nu(B_k)|.
\end{align*} 
By summing over $S\in \SS_{k,j-1}$, we obtain that
\begin{equation}\label{eq:step2}
\begin{split}
&m_k\sum_{S\in \SS_{k,j-1}} |\mu_k(S)-\nu_k(S)|  \\
&\leq |\mu(B_k)-\nu(B_k)| + \sum_{S\in \SS_{k,j-1}}|\mu(S) - \nu(S)|.
\end{split}
\end{equation}
Using $\sum_{j=1}^J 2^{-pj} \leq 2^{-p}/(1-2^{-p})$, and putting together inequalities \eqref{eq:cost_transport1}, \eqref{eq:lemma_from_singh} and \eqref{eq:step2}, one obtains the inequality of Lemma \ref{lem:deterministic_control}.
\end{proof}
\begin{figure}
    \centering
    \includegraphics[width=0.7\columnwidth]{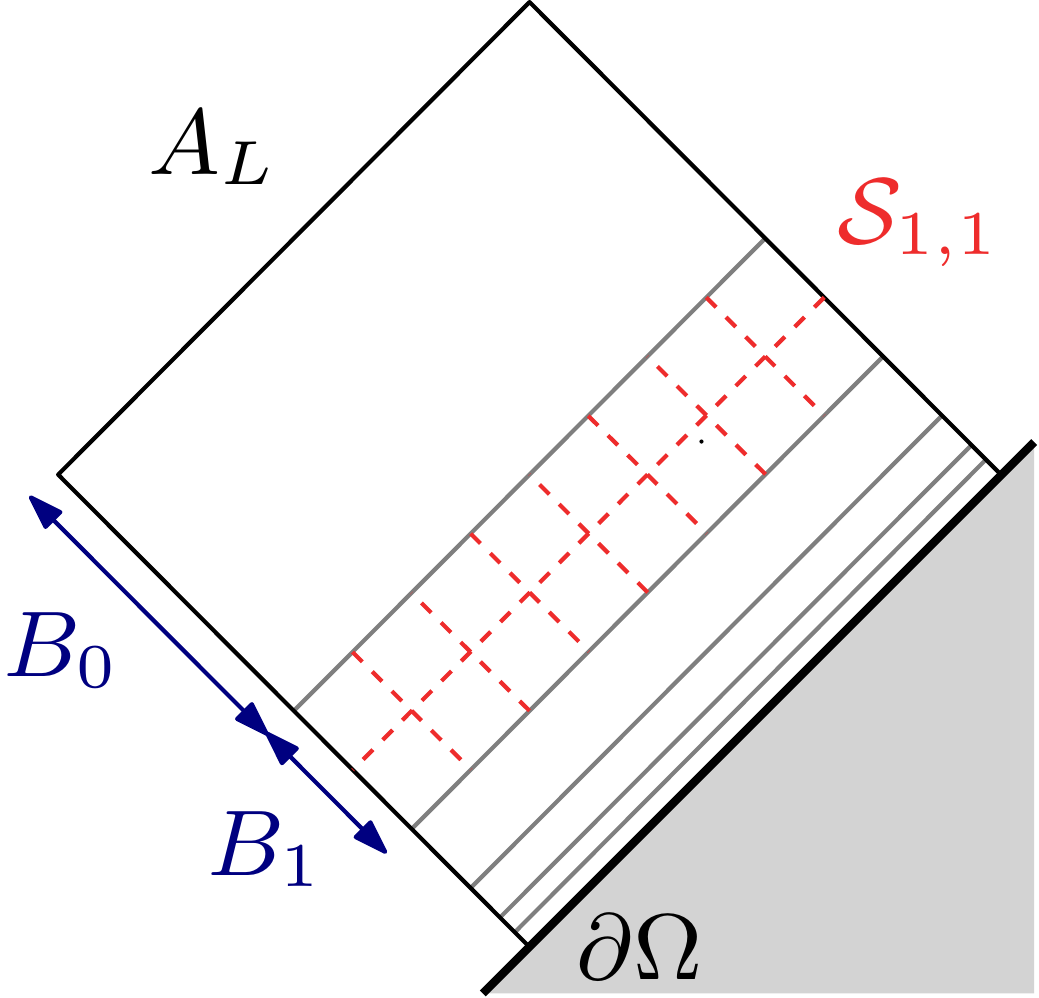}
    \caption{Partition of $A_L$ used in the proof of Theorem \ref{thm:minimax_result}}
    \label{fig:partition_box}
\end{figure}

Before proving \cref{thm:minimax_result}, we state a useful inequality. Let $\mu\in \MM^q_{M,L}$ and let $B\subset \Omega$ be at distance $\ell$ from the diagonal $\thediag$. Then,
\begin{equation}\label{eq:mass_far_from_diag}
    \mu(B) = \int_{B} \frac{\|x-\thediag\|^q}{\|x-\thediag\|^q} \dd \mu(x) \leq M\ell^{-q}.
\end{equation}

\begin{proof}[Proof of Theorem \ref{thm:minimax_result}]
Consider a distribution $P\in\PP^q_{M,L}$. Remark first that for any measure $\mu\in \MM^q_{M,L}$, we have $\mu(B_k) \leq ML^{-q}2^{kq}$  one  by \eqref{eq:mass_far_from_diag}. 
Let $\mu$ be a random persistence measure of law $P$ and $\overline\mu_n$ be the empirical EPD associated to a $n$-sample of law $P$. By the Cauchy-Schwartz inequality, given a Borel set $A\subset \upperdiag$, we have
\begin{equation}\label{eq:cauchy}
\E|\overline\mu_n(A)-\EPD(A)|  \leq \sqrt{\frac{\E[\mu(A)^2]}{n}}.
\end{equation}
The Cauchy-Schwartz inequality also yields, as $|\SS_{k,j-1}|= 2^{k+1}4^{j-1}$,
\begin{align*}
&\sum_{S\in\SS_{k,j-1}} \E|\hat \mu_n(S)-\EPD(S )| \leq \sum_{S\in\SS_{k,j-1}} \sqrt{\frac{\E[\mu(S)^2]}{n}} \\
&\qquad\leq  \sqrt{\frac{\E\left[\sum_{S\in\SS_{k,j-1}}\mu(S)^2\right]}{n} |\SS_{k,j-1}|}   \\
&\qquad\leq \sqrt{\frac{\E\left[\mu(B_k)^2\right]}{n} |\SS_{k,j-1}|}  \leq \frac{ML^{-q}2^{kq}}{\sqrt{n}} 2^{\frac{k+1}{2}}2^{j-1}.
\end{align*}
Note also that $\sum_{S\in\SS_{k,j-1}} \E|\hat \mu_n(S)-\EPD(S )|\leq 2\EPD( B_k)\leq 2ML^{-q}2^{kq}$ and that $\overline \mu_n(B_k)\wedge \EPD(B_k)\leq ML^{-q}2^{kq}$. By using those three previous inequalities, Lemma \ref{lem:deterministic_control} and inequality \eqref{eq:cauchy}, we obtain that $\E[\Dp^p(\overline\mu_n,\EPD)]$ is smaller than
\begin{align*}
    & 2^{p/2}ML^{p-q}\sum_{k\geq 0}2^{-kp}\Big(2^{-Jp}2^{kq} + \frac{c_p}{\sqrt{n}} 2^{kq} \\
&\qquad+ \sum_{j=1}^J 2^{-jp}2^{kq}\p{2\wedge \frac{ 2^{\frac{k+1}{2}}2^{j-1}}{\sqrt{n}} }  \Big) \\
&\leq c_{p,q}ML^{p-q}\Big(2^{-Jp} + \frac{1}{\sqrt{n}}+U\Big),
\end{align*}
where $U=\sum_{k\geq 0}\sum_{j=1}^J 2^{k(q-p)}2^{-jp}\p{1\wedge \frac{2^{\frac{k}{2}}2^j}{ \sqrt{n }}}$. To bound $U$, we remark that if $k\geq \log_2(n)$, then the minimum in the definition of $U$ is equal to $1$. Therefore, letting $b_J=1$ if $p>1$ and $b_J=J$ if $p=1$, we find that $U$ is smaller than
\begin{align*}
   & \sum_{k=0}^{\log_2(n)}\sum_{j=1}^J \frac{2^{k(q-p+1/2)}2^{(1-p)j}}{\sqrt{n}} + \sum_{k\geq \log_2(n)} \sum_{j=1}^J2^{-kp}2^{-jp}\\
    &\qquad\leq c_pb_J\sum_{k<\log_2(n)}\frac{2^{k(q+1/2-p)}}{\sqrt{n}} + c_p n^{-p} \\
    &\qquad\leq c_{p,q} b_J(n^{-1/2}\vee n^{q-p}).
\end{align*}
Eventually, if $p>1$, we may set $J=+\infty$ and obtain a bound of order $ML^{p-q}(n^{-1/2}+n^{q-p})$. If $p=1$, we choose $J= (q-p)(\log n)/(2p)$ to obtain a rate of order $n^{-1/2} + n^{q-p}\log n$.
\end{proof}

\begin{figure}
    \centering
    \includegraphics[width=0.7\columnwidth]{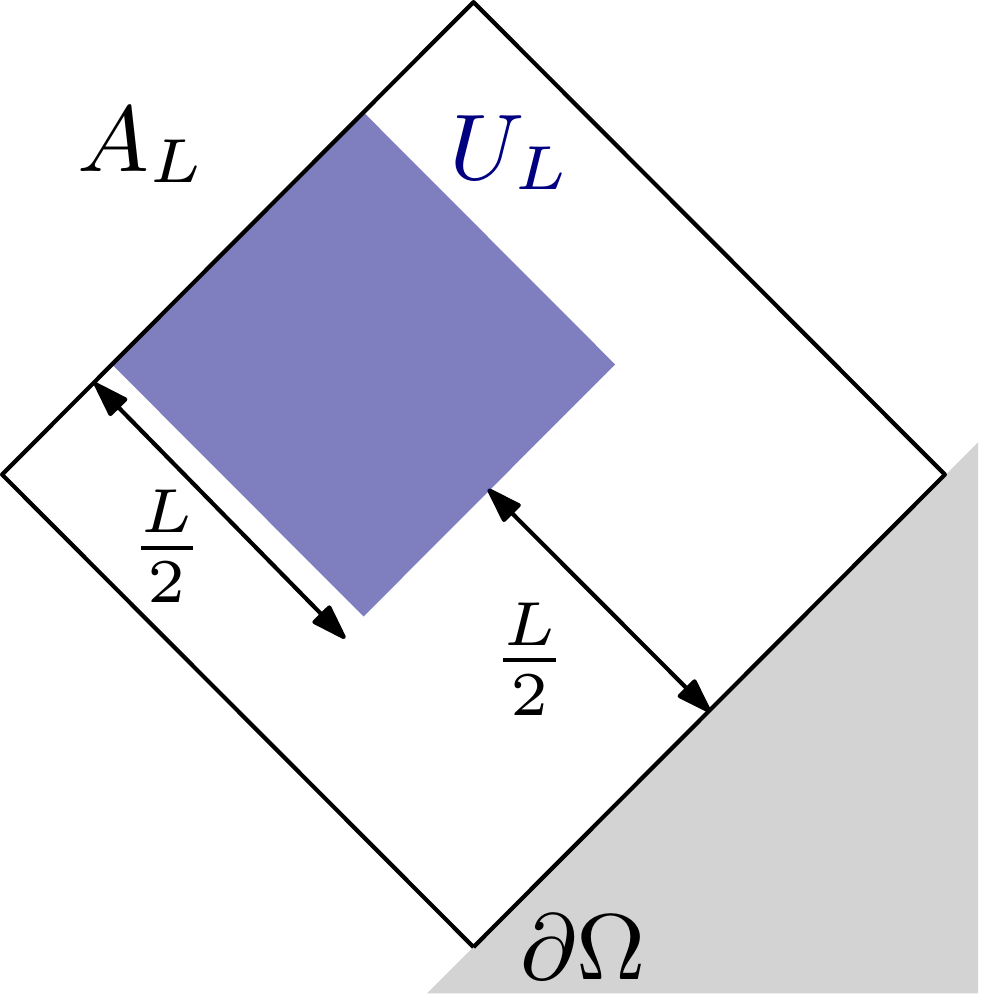}
    \caption{In the box $U_L$, the distance $\rho$ is equal to the Euclidean distance.}
    \label{fig:box_ul}
\end{figure}

\begin{proof}[Proof of Theorem \ref{thm:minimax}]
    As $\PP^{q,s}_{L,M,T}\subset\PP^q_{L,M}$, we have $\mathcal{R}_n(\PP^q_{L,M})\geq \mathcal{R}_n(\PP^{q,s}_{L,M,T})$. Therefore, Theorem \ref{thm:regularityEstimation}, whose proof is found below, directly implies Theorem \ref{thm:minimax}.
\end{proof} 

\begin{proof}[Proof of Theorem \ref{thm:regularityEstimation}]
We first consider the case $q=0$. If $\mu,\nu$ are two measures on $\upperdiag$ of mass smaller than $M$, then $\Dp(\mu,\nu) = W_{p,\rho}(\Phi(\mu),\Phi(\nu))$ \citep[Prop.~3.15]{tda:divol2019understanding}, where $\rho$ is the distance on $\tilde\upperdiag\defeq \upperdiag \cup \{\thediag\}$ defined by $\forall x,y\in \tilde\upperdiag,$
\[
\rho(x,y)=\min(\|x-y\|,d(x,\thediag)+d(y,\thediag))\]
and $\Phi(\mu)=\mu + (2M-|\mu|)\delta_{\thediag}$. Remark that $\rho(x,y)=\|x-y\|$ if $x,y\in U_L$, where $U_L\subset A_L$ is any $\ell_1$-ball of radius $L/\sqrt{8}$ at distance $L/2$ from the diagonal, see Figure \ref{fig:box_ul}. As $\Phi$ is a bijection, the minimax rates for the estimation of $\EPD$ is therefore equal to 
\[ \inf_{\Phi(\hat{\mu}_n)}\sup_{P\in \PP^{0,s}_{L,M,T}} \E[W_{p,\rho}^p(\Phi(\hat{\mu}_n),\Phi(\EPD))].\]
Let $\QQ$ be the set of probability measures on $U_L$ whose densities belong to $B^s_{p',q'}$ with associated norm smaller than $T/M$. Then, $\PP_{M,L,T}^{0,s}$ contains in particular the set of all distributions $P$ for which $\mu\sim P$ satisfies $\Phi(\mu) =M\delta_{x}$ and $x$ is sampled according to some law $\tau\in \QQ$. For such a distribution $P$, one has $\Phi(\EPD)=M\tau$, so that the minimax rate is larger than
\[ \inf_{\hat{a}_n}\sup_{\tau\in \QQ} \E[W_p^p(\hat{a}_n,M\tau)],\]
where the infimum is taken on all measurable functions based on $K$ observations of the form $M\delta_{x_i}$ with $x_1,\dots,x_n$ a $n$-sample of law $\tau \in \QQ$. Hence, we have shown that the minimax rate for the estimation of $\EPD$ with respect to $\Dp$ is larger up to a factor $M$ than the minimax rate for the estimation of $\tau \in \QQ$ given $n$ i.i.d.~observations of law $\tau$. As the minimax rate for this problem is known to be larger than $L^p/\sqrt{n}$ \citep[Thm.~5]{ot:weed2019estimation}, we obtain the conclusion in the case $q=0$.

For the general case $q>0$, we remark that if $M' = ML^{-q}$ then $\PP^{0,s}_{M',L}$ is included in $\PP^{q,s}_{M,L,T}$. In particular, the minimax rate on $\PP^{q,s}_{M,L,T}$ is larger than the minimax rate on $\PP^{0,s}_{M',L,T}$, which is larger than $c\frac{M'L^{p}}{\sqrt{n}}=c\frac{ML^{p-q}}{\sqrt{n}}$ for some constant $c>0$.
\end{proof}

\begin{remark}[Case $p=\infty$]
    It can be shown that for $p = \infty$, the minimax rate is larger than $ c_a n^{-a}$, $\forall a>0$. This is a consequence of an inequality between the $\OT_\infty$ distance and the distance between the support of the measures, for which minimax rates are known \citep{hardle1995estimation}. This means that no reasonable estimator exists on $\mathcal{P}{L,M}^\infty$: some additional conditions should be added, while standard assumptions in the support estimation literature seem artificial in our context (as in \cref{rem:rip_estimation}).
\end{remark}

\section{Delayed proofs from \cref{subsec:quantiz_generalities}}

\begin{proof}[Proof of \cref{lemma:restrictToCodebook}]
 Fix a codebook $\bc = (c_1 \dots c_k)$. Let $T_\bc : x \mapsto c_j$ if $x \in V_j(\bc)$ ($1 \leq j \leq k$) and $\projdiag(x)$ if $x \in V_{k+1}(\bc)$, where $\projdiag(x)$ denotes the orthogonal projection of a point $x \in \upperdiag$ on the diagonal $\thediag$. 
 Let $\pi$ be the pushforward of $\mu$ by the map $x \mapsto (x,T_\bc(x))$, extended on $\groundspace \times \groundspace$ by $\pi(U, \groundspace) = 0$ for $U \subset \thediag$ (intuitively, $\pi$ pushes the mass of $\mu$ on their nearest neighbor in $\{c_1 \dots c_{k+1} \}$).
 One has, for $A, B \subset \upperdiag$, $\pi(A, \groundspace) = \mu((\id, T_c)^{-1}(A, \groundspace)) = \mu(A)$, and $\pi(\groundspace, B) = \mu(T_c^{-1}(B)) = \sum_j \mu(V_j(\bc)) \ones\{c_j\in B\}$, that is $\pi$ is an admissible between the measures $\mu$ and $\sum_j \mu(V_j(\bc)) \delta_{c_j}$. Hence,
 \[
    \Dp^p\hspace{-.1cm}\p{\hspace{-.05cm} \mu, \sum_j \mu(V_j(\bc)) \delta_{c_j}\hspace{-.05cm} }\hspace{-.1cm} \leq \int_{\groundspace}  \min_{1 \leq j \leq k+1} \| x - c_j\|^p \dd \mu(x).
 \]
Let $(m_1 \dots m_k)$ be a vector of non-negative weights, let $\nu = \sum_{j=1}^k m_j \delta_{c_j}$, and $\pi$ be an admissible transport plan between $\mu$ and $\nu$. 
 One has
 \begin{align*}
    \int_{\groundspace \times \groundspace} \|x - y \|^p \dd \pi(x,y) &= \sum_{j=1}^{k+1} \int_{\groundspace} \|x - c_j\|^p \dd \pi(x,c_j) \\
    \geq &\sum_{j=1}^{k+1} \int_{\groundspace} \min_{j'} \|x - c_{j'}\|^p \dd \pi(x, c_j) \\
    \geq &\int_{\groundspace} \min_{j'} \|x - c_{j'}\|^p \dd \mu(x) \\
    \geq &\Dp^p\left( \mu, \sum_{j=1}^k \mu(V_j(\bc)) \delta_{c_j} \right).
 \end{align*}
Taking the infimum over $\pi$ gives the conclusion.
\end{proof}

We now turn to the proof of Proposition \ref{prop:existence_and_prop_of_minimizers}. For technical reasons, we extend the function $R_k$ to $\groundspace^k$, by noting that if $c_j\in \thediag$, then the Vorono\"i cell $V_j(\bc)$ is empty by definition, see \eqref{eq:def_voronoi}.

\begin{lemma}\label{lemma:centroids_are_distincts}
Let $\bc \in \groundspace^k$ be such that there exists $1\leq j \leq k$ with $\mu(V_j(\bc^*))=0$. Then, $R_k(\bc)>R_k^*$.
\end{lemma}
In particular, if two centroids of a codebook $\bc$ are equal or if a centroid $c_j$ of $\bc$ belongs to $\thediag$, then the condition of the above lemma is satisfied, so that the $\bc$ cannot be optimal. This proves the second part of \cref{prop:existence_and_prop_of_minimizers}.

\begin{proof}[Proof of \cref{lemma:centroids_are_distincts}]
Let $\bc=(c_1,\dots,c_k)\in \groundspace^k$. 
Assume without loss of generality that  $\mu(V_{1}(\bc)) = 0$. 
Let $\bc_0 = (c_2, \dots, c_k) \in \groundspace^{k-1}$ (that is, $\bc$ where we removed the first centroid).
Assume first that $\mu(V_{k+1}(\bc)) > 0$, that is there is some mass transported onto the diagonal. 
Consider a compact subset $A \subset V_{k+1}(\bc)$ such that $\mu(A) > 0$ and the diameter $\diam(A)$ of $A$ is smaller than the distance $d(A,\thediag)$ between $A$ and $\thediag$. 
Let $c'\in A$ and observe that, for $x\in A$, $\|x-c'\|< \|x-\thediag\|$. Therefore, 
\[ \int_A \|x - c'\|^p \dd \mu(x) < \int_A \|x - \thediag\|^p \dd \mu(x). \]
Consider the measure $\nu = \hat{\mu}(\bc_0) + \mu(A) \delta_{c'}$. Then 
\begin{align*} 
&\Dp^p(\nu,\mu) \leq \sum_{j=1}^{k} \int_{V_j(\bc)} \|x - c_j\|^p \dd \mu(x) \\
&+ \int_{V_{k+1}(\bc) \backslash A} \| x - \thediag\|^p \dd \mu(x) + \int_A \|x - c'\|^p \dd \mu(A) \\
< &R_k(\bc),
\end{align*}
thus $\bc$ cannot be optimal. We can thus assume that $\mu(V_{k+1}(\bc)) = 0$, in which case we can reproduce the proof of \citep[Thm~4.1]{graf2007foundations}, which gives that $\bc$ cannot be optimal either in that case, yielding the conclusion. 
\end{proof}

\begin{lemma}\label{lemma:RkContinuous}
$R_k$ is continuous.
\end{lemma}

\begin{proof}[Proof of \cref{lemma:RkContinuous}]
	For a given $x \in \groundspace$, the map $\bc \mapsto \min_i \|x - c_i\|^p$ is continuous and upper bounded by $\|x - \thediag\|^p$. Thus, $R_k$ is continuous by dominated convergence as we have finite $\Pers_p$.
\end{proof}

\begin{lemma}\label{lemma:aux_minimizer}
Let $0\leq \lambda < R_{k-1}^*$. Then, the set $\{\bc\in \groundspace^k,\ R_k(\bc) \leq \lambda \}$ is compact.
\end{lemma}

\begin{proof}[Proof of \cref{lemma:aux_minimizer}] 
Fix $\lambda < R_{k-1}^*$. 
The set is closed by continuity of $R_k$, so that it suffices to show that it is bounded. Let $\bc$ be such that $R_k(\bc) \leq \lambda$. 
Pick $L$ such that $\int_{A_L} \|x - \thediag\|^p \dd \mu(x) \geq \lambda$ and $\int_{A_{L}^c} \|x - \thediag\|^p \dd \mu(x) < R_{k-1}^* - \lambda$. 
Such a $L$ exists since $\int_\upperdiag \|x - \thediag\|^p \dd \mu(x) = \Pers_p(\mu)=R_0^*\geq R_{k-1}^*$. 
Then, all the $c_j$s must be in $A_{2L}$. Indeed, assume without loss of generality that $c_1 \in A_{2L}^c$. Then $V_1(\bc) \subset A_{L}^c$, as any point in $A_{L}$ is closer to the diagonal than to $c_1$. Therefore,
\begin{align*}
	R_{k-1}^* \leq &\sum_{j=2}^{k+1} \int_{V_j(\bc)} \| x - c_j\|^p \dd \mu(x) \\
	&+ \int_{V_1(\bc)} \min_{j \in \{2 \dots k+1\}} \|x - c_j\|^p \dd \mu(x) \\
	\leq &R_k(\bc) + \int_{V_1(\bc)} \|x - \thediag\|^p \dd \mu(x) \\
	\leq &R_k(\bc) + \int_{A_L^c} \| x - \thediag\|^p \dd \mu(x) \\
	< &\lambda + R_{k-1}^* - \lambda = R_{k-1}^*,
\end{align*}
leading to a contradiction.
\end{proof}

\begin{proof}[Proof of \cref{prop:existence_and_prop_of_minimizers}]
We show by recursion on $0\leq m \leq k$ that $R_m^* <R_{m-1}^*$ and that $\bC_m$ is a non-empty compact set (with the convention $R_{-1}^*=+\infty$. The initialization holds as $R_0^*= \Pers_p(\mu)<+\infty$ with the empty codebook being optimal. We now prove the induction step. Let $\bc =(c_1,\dots,c_{m-1})\in \bC_{m-1}$. Consider $\bc'=(c_1,c_1,c_2,\dots,c_{m-1})$. 
Then, $\mu(V_1(\bc'))=0$, so that $R_{m-1}^* = R_{m-1}(\bc)=R_m(\bc')>R_m^*$ by Lemma \ref{lemma:centroids_are_distincts}. Furthermore, pick $\lambda \in (R_m^*, R_{m-1}^*)$. Then, $R_m^*$ is equal to the infimum of $R_m$ on the set $\{\bc\in \groundspace^k,\ R_m(\bc) \leq \lambda \}$, which is compact according to Lemma \ref{lemma:aux_minimizer}. As the function $R_k$ is continuous, the set of minimizers $\bC_m$ is a non-empty compact set, concluding the induction step.
\end{proof}

\begin{proof}[Proof of \cref{coro:positive_quantities}]
The quantities being minimized in the definitions of $\Dmin$ and $\mmin$ are both continuous functions of $\bc^*$. As the set $\bC_k$ is compact, the minima are attained, and cannot be equal to $0$ according to \cref{prop:existence_and_prop_of_minimizers}. 
\end{proof}

\section{Proof of \cref{thm:online_quantiz}.}
In the following, we fix a distribution $P$ supported on $\MM^p_{L,M}$ and we consider $\bc^*$ be an optimal codebook of $\EPD$. 
The different constants encountered in this section all depend on the parameters $p,L, M, k, \Dmin$ and $\mmin$. In particular, we introduce the quantity
\begin{equation*}
    \begin{split}
        &\mmax \defeq \sup_{\mu\in \MM_{L,M}^p}\sup_{1\leq j\leq k} \mu(V_j(\bc^*)).
    \end{split}
\end{equation*}
 Note that 
 $\mmax  \leq \frac{2^pM}{\Dmin^p}$ as $\int_{V_j(\bc^*)}\dd\mu(x) \leq \frac{2^p}{\Dmin^{p}}\int_{V_j(\bc^*)} \| x - \thediag\|^p\dd\mu(x)$. 
 
The proof of Theorem \ref{thm:online_quantiz} follows the proof of \citep[Thm.~5]{chazal2020optimal}. As a first step, we show that it is enough to prove the following lemma, which relates the loss of $\bc^{(t)}$ and the loss of $\bc^{(t+1)}$. 

\begin{lemma}\label{lem:key_lemma}
There exists $R_0>0$ such that, if $\|c^{(0)}_j -c_j^*\|\leq R_0$ for $1\leq j\leq k$, then
\[ \E\| \bc^{(t+1)} - \bc^* \|^2 \leq \p{1-\frac{C_0}{t+1}}\E\| \bc^{(t)} - \bc^* \|^2 + \frac{C_1}{(t+1)^2},\]
for some constants $C_0>1$, $C_1>0$.
\end{lemma}

\begin{proof}[Proof of \cref{thm:online_quantiz}]
    From Lemma \ref{lem:key_lemma}, we show by induction that $u_t \defeq \E\| \bc^{(t)} - \bc^* \|^2$ satisfies $u_t \leq \frac{\alpha}{t+1}$ for $\alpha=C_1/(C_0-1)$. This concludes the proof as $T$ is of order $n/\log(n)$. 
    The initialization holds by assumption as long as $R_0\leq \alpha$, whereas we have by induction 
    \begin{align*}
            u_{t+1}&\leq \p{1-\frac{C_0}{t+1}}\frac{\alpha}{t+1} + \frac{C_1}{(t+1)^2}\\
            &\leq \frac{\alpha}{(t+1)^2}\p{t+1-C_0+C_1/\alpha} =\frac{\alpha t}{(t+1)^2},
    \end{align*}
    which is smaller than $\alpha/(t+2)$.
\end{proof}
The proof of Lemma \ref{lem:key_lemma} is a close adaptation  of   \citep[Lemma 21]{chazal2020optimal}. 
The proof of the latter contains tedious computations (that we do not reproduce here) which can be adapted \textit{mutatis mutandis} to our setting once the two following key results are shown. 
Given a codebook $\bc$, we let $p_j(\bc) = \EPD(V_j(\bc))$ and similarly, given a $n$-sample $\mu_1,\dots,\mu_n$ of law $P$, we let $\hat p_j(\bc)= \overline\mu_n(V_j(\bc))$. 
Note that if $\|\bc - \bc^*\|$ is small enough, one has $p_j(\bc) \leq 2 \mmax$.
Also, we let $w_p(\bc,\mu)_j\defeq \mu(V_j(\bc))v_p(\bc,\mu)_j$ for $\mu\in \MM^p$ and $1\leq j\leq k$. 
Recall that we assume that the EPD $\EPD$ satisfies the margin condition (\cref{def:margin}) with parameters $\lambda$ and $r_0$ around the optimal codebook $\bc^*$.

\begin{lemma}[Lemma 22 in \citep{chazal2020optimal}]\label{lem:key_lemma1}
	Let $R_0$ be small enough respect to $r_0 \Dmin^2/L^2$ and 
	let $\bc$ be such that $\| \bc - \bc^*\| \leq R_0$. Then, we have
	\begin{equation*}
\sum_{j=1}^k |p_j(\bc) - p_j(\bc^*)| \leq 2 \lambda r_0,
	\end{equation*}
	and 
	\[ \| w_2(\bc,\EPD)-w_2(\bc^*,\EPD)\| \leq 7\sqrt{2}\lambda\frac{ L^3}{\Dmin^2}\|\bc-\bc^*\|.\]
	\label{lemma:gradient_lipschitz}
\end{lemma}
As $w_2(\bc^*,\EPD)_j=p_j(\bc^*)\bc^*_j$, Lemma \ref{lem:key_lemma1} indicates that the application $w_2(\cdot,\EPD)$ is Lipschitz continuous around an  optimal codebook $\bc^*$, a key property to show the convergence of the sequence $(\bc^{(t)})_t$. 

\begin{lemma}[Lemma 24 in \citep{chazal2020optimal}]\label{lem:key_lemma2}
	Let $\bc$ be a codebook such that $\hat{p}_j(\bc) \leq 2 \mmax$ (which is always possible if $\|\bc - \bc^*\|$ is small enough). 
	Then, with probability larger than $1 - 2 k e^{-x}$, we have, for all $1\leq j\leq k$,
	\begin{equation}\label{eq:bernstein}
	|\hat{p}_j(\bc) - p_j(\bc)| \leq \sqrt{\frac{4 \mmax p_j(\bc) x}{n}} + \frac{2\mmax x}{3n}.
	\end{equation}
	Moreover, with probability larger than $1 - e^{-x}$, we have
	\begin{equation}\label{eq:transportoncell}
	\|w_2(\bc,\overline\mu_n)-w_2(\bc,\EPD)\|\leq 2 \mmax L \sqrt{\frac{2k}{n}} \p{ 1 + \sqrt{\frac{x}{2}}}.
	\end{equation}
\end{lemma}
The proof of this lemma follows from standard concentration inequalities.

\begin{proof}[Proof of Lemma \ref{lem:key_lemma2}]
	Equation \eqref{eq:bernstein} follows from Bernstein inequality applied to the real-valued random variable $0 \leq \hat p_j(\bc) \leq 2 \mmax$, with variance bounded by 
$\E[\mu(V_j(\bc))^2]/n \leq \mmax p_j(\bc)/n.$
	
	For equation \eqref{eq:transportoncell}, we introduce the function $f_j:x\mapsto x\ones\{x\in V_j(\bc)\}$, so that $w_2(\bc,\mu)_j=\mu(f_j)$, the integral of $f_j$ against $\mu$. We have $w_2(\bc,\mu_n)_j-w_2(\bc,\EPD)_j= n^{-1}\sum_{i=1}^n(\mu_i(f_j)-\EPD(f_j))$. Note that $\|\mu_i(f_j)\|\leq \sqrt{2} L  \cdot 2\mmax$. We write
	\begin{align*}
	    &\E\left\| \frac{1}{n}\sum_{i=1}^n(\mu_i(f_j)-\EPD(f_j))_j\right\|\\
	    &\leq \sqrt{\E\left\| \frac{1}{n}\sum_{i=1}^n(\mu_i(f_j)-\EPD(f_j))_j\right\|^2} \\
	    &\leq \sqrt{\frac{1}{n}\E \left\|(\mu_1(f_j))_j\right\|^2} \leq 2\sqrt{\frac{k}{n}}\sqrt{2}L\mmax.
	\end{align*}
	Also, note that $F(\mu_1,\dots,\mu_n)=\|w_2(\bc,\mu_n)-w_2(\bc,\EPD)\|$ satisfies a bounded difference condition of parameter $4\sqrt{2}L\mmax$ \citep[Sec.~6.1]{boucheron2013concentration}. 
	A bounded difference inequality \citep[Thm.~6.2]{boucheron2013concentration} yields the result.
\end{proof}

\begin{figure}
    \centering
    \includegraphics[width=0.7\columnwidth]{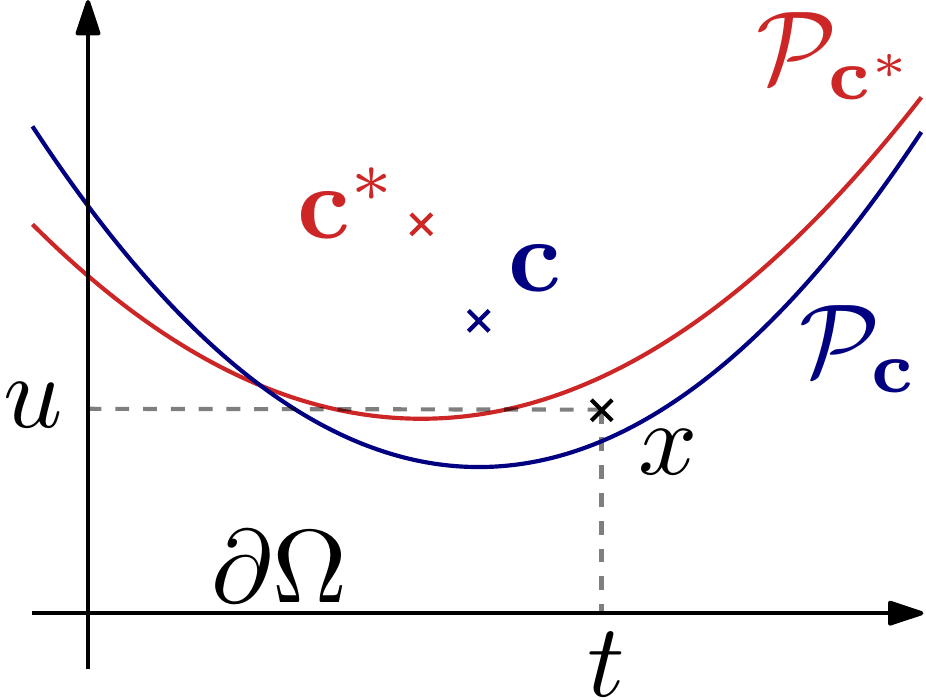}
    \caption{Illustration of the proof of Lemma \ref{lemma:gradient_lipschitz}}
    \label{fig:parabola}
\end{figure}

The proof of \cref{lemma:gradient_lipschitz} relies on the following lemma, that essentially tells that the area of misclassified points when using a codebook $\bc$ instead of an optimal one $\bc^*$ can be controlled linearly in terms of $\|\bc^* - \bc\|$. Note that this result is well-known when boundaries between the cells are hyperplanes (as it is the case in standard quantization), it remains to treat the case when the boundary is a parabola. Let $d(x,A)$ be the distance from a point $x\in \Omega$ to $A\subset \Omega$.
\begin{lemma}
	Let $\bc^*$ be an optimal codebook, and $\bc \in A_L^k$. Let $x \in A_L$ and $1\leq j\leq k$. Assume that $x \in V_j(\bc^*) \cap V_{k+1}(\bc)$. Then, $d(x, \partial V_j(\bc^*)) \leq \frac{7L^2}{2\Dmin^2}\| \bc^* - \bc\|$.
	Symmetrically, if $x \in V_{k+1}(\bc^*) \cap V_j(\bc)$, one has $d(x, \partial V_{k+1}(\bc^*)) \leq \frac{7L^2}{2\Dmin^2}\| \bc^* - \bc\|$.
	\label{lemma:dependence_cells_on_centroid}
\end{lemma}

\begin{figure*}
    \centering
    \includegraphics[width=0.8\textwidth]{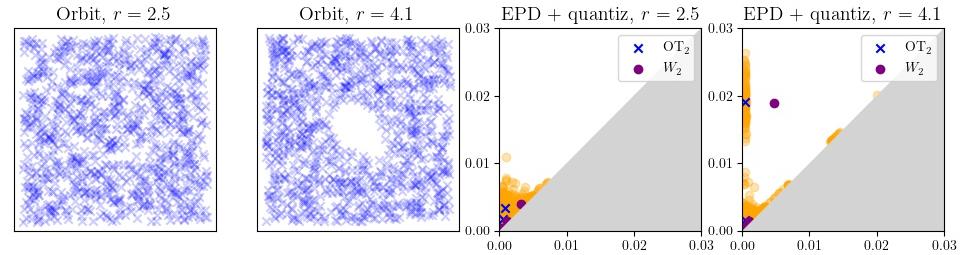}
    \caption{(Left) Two observations of the \texttt{ORBIT5K} dataset from two different classes (whose dynamics depend on a parameter $r$, see \citep{tda:adams2017persistenceImages} for details). (Right) The empirical EPD (orange) observed for these two classes and the corresponding quantization obtained using our $\OT_2$ algorithm with $k=2$ and the $W_2$ algorithm \citep{chazal2020optimal} with $k=3$. As we account for the diagonal in a natural geometric way in our formulation, our quantization reflects the structure of the empirical EPD in a better way. This is especially striking in the case $r=4.1$ (most right plot) where a centroid for the $W_2$ algorithm is deviated to a peculiar position due to the presence of few points close to the diagonal. Such points belong to the diagonal cell $V_{k+1}$ in our setting.}
    \label{fig:orbits}
\end{figure*}

\begin{proof}[Proof of \cref{lemma:dependence_cells_on_centroid}]
For convenience, we write in this proof the coordinates of points in the basis $(\thediag,\thediag^\perp)$, that $x\in \upperdiag$ will have coordinates $(a,b)$ where $a$ is the projection of $x$ on $\thediag$ and $b=\|x-\thediag\|$. Also, given $y=(a,b)\in\upperdiag$, we let $\mathcal{P}_{y}$ be the parabola with focus $y$ and directrix $\thediag$. To put it another way, if $y=(a,b)$, then $\PP_y$ is the image of $\thediag$ by the map
\[ f(a,b,\cdot):t \mapsto \frac{(t-a)^2}{2b} + \frac{b}{2}. \]
 One can check that for all $t\in [-L/2,L/2]$, if $b=\|y-\thediag\|\geq \Dmin$, we have $\left| \frac{\partial f}{\partial a} \right| \leq \frac{L}{\Dmin}$ and $\left| \frac{\partial f}{\partial b} \right| \leq \frac{1}{2} + \frac{(t-a)^2}{b}\frac{1}{b} \leq \frac{1}{2} + \frac{2L^2}{\Dmin^2}$.
 
Let $c_j^*=(a^*,b^*)$ and $c_j=(a,b)$.
 Let $x = (t,u) \in V_j(\bc^*)\cap V_{k+1}(\bc)$. Then, $u\geq f(a^*,b^*,t)$, whereas $u\leq f(a,b,t)$. The distance $d(x,\partial V_j(\bc^*))$ is smaller than $u-f(a^*,b^*,t)$
\begin{align*}
&u-f(a^*,b^*,t) \leq f(a,b,t)-f(a^*,b^*,t) \\
&\leq |f(a^*, b^*, t) - f(a, b^*, t)| + |f(a, b^*, t) - f(a, b, t)| \\
&\leq \int_{a \wedge a^*}^{a \vee a^*} \left| \frac{\partial f}{\partial a}(\alpha, b^*, t) \right| \dd \alpha + \int_{b \wedge b^*}^{b \vee b^*} \left| \frac{\partial f}{\partial b}(a, \beta, t) \right| \dd \beta \\
&\leq \frac{L}{\Dmin}|a - a^*| + \p{\frac{1}{2}+\frac{2L^2}{\Dmin^2}}|b - b^*|  \\
&\leq \p{\frac{1}{2} + \frac{L}{\Dmin} + \frac{2L^2}{\Dmin^2}} \| \bc - \bc^*\| \leq \frac{7}{2} \frac{L^2}{\Dmin^2} \| \bc - \bc^*\|,
\end{align*}
which proves the claim.
\end{proof}

\begin{proof}[Proof of \cref{lemma:gradient_lipschitz}]\renewcommand{\qedsymbol}{}
This proof is inspired from \citep[Appendix A.3]{levrard2015nonasymptotic}. 
Let us prove the first point. One has, with $t=\frac{7L^2}{2\Dmin^2} \| \bc - \bc^* \|\leq r_0$,
\begin{align*}
	&\sum_{j=1}^k |p_j(\bc) - p_j(\bc^*)| = \sum_{j=1}^k \left|\EPD(V_j(\bc))-\EPD(V_j(\bc^*) \right| \\
							&\qquad\leq 2 \sum_j \sum_{j' \neq j} \EPD( V_j(\bc) \cap V_{j'}(\bc^*) ) \\
							&\qquad \leq 2 \EPD[ N(\bc^*)^t] \leq 2 \lambda t \leq 2\lambda r_0.
\end{align*}
where we applied \cref{lemma:dependence_cells_on_centroid} and the margin condition. To prove the second inequality, remark that $w_2(\bc,\EPD)_j = \int_{V_j(\bc)} x\dd \EPD(x)$. Therefore,
\begin{align*}
&\|w_2(\bc,\EPD)-w_2(\bc^*,\EPD)\| \\
&\leq \sum_{j=1}^k \|w_2(\bc,\EPD)_j-w_2(\bc^*,\EPD)_j\|\\
&\leq \sum_{j=1}^k \left \| \int_{V_j(\bc)} x\dd \EPD(x) - \int_{V_j(\bc^*)} x\dd \EPD(x)\right \|\\
&\leq  2 \sum_j \sum_{j' \neq j} \int_{V_j(\bc) \cap V_{j'}(\bc^*)}\|x\|\dd \EPD(x)\\
&\leq 2\sqrt{2}L\lambda t\leq 7\sqrt{2}\lambda\frac{ L^3}{\Dmin^2}\|\bc-\bc^*\|. && \square
\end{align*}
\end{proof}

\section{Complementary experiments}

\begin{figure}
    \centering
    \includegraphics[width=0.9\columnwidth]{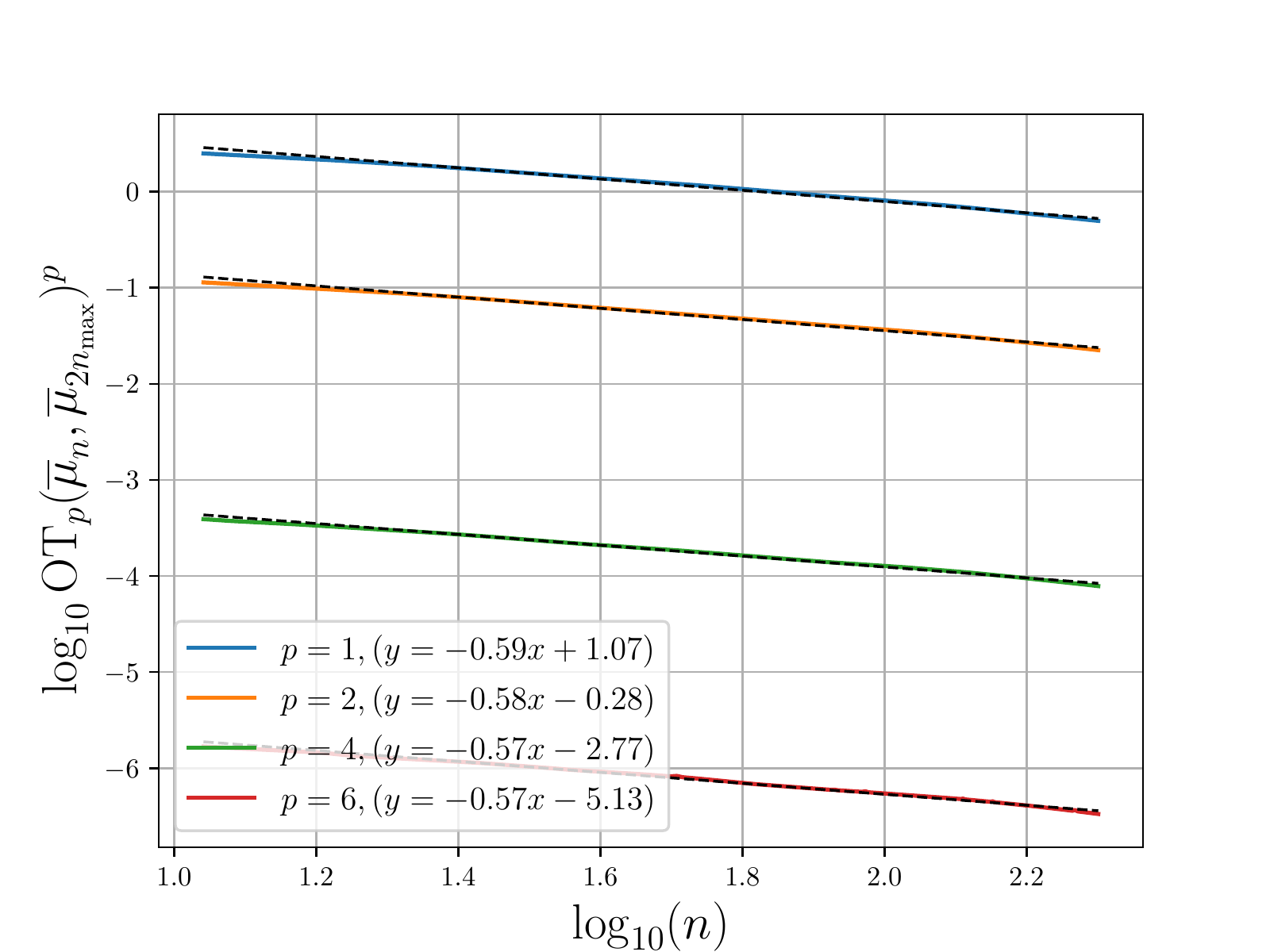}
    \caption{Convergence for $p=1,2,4,6$, each exhibiting a rate $\sim n^{-1/2}$.}
    \label{fig:cv_rate}
\end{figure}

\cref{fig:cv_rate} showcases the convergence rate of the empirical EPD using the same setting as in \cref{fig:torus_expe} (points sampled on the surface of the torus) for different values of $p$, each of them exhibiting a rate of $n^{-1/2}$.

We perform a complementary experiment on the \texttt{ORBIT5K} dataset \citep[\S 6.4.1]{tda:adams2017persistenceImages}, a benchmark dataset in TDA made of $5$ classes with 1000 observations each (split into 70\%/30\% training/test) representing different dynamical systems, turned into PDs through \v Cech filtrations. For each class $i \in \{1,\dots,5\}$, we compute a $2$-quantization $\nu^{(i)}$ using our $\mathrm{OT}_2$ algorithm and a $3$-quantization $\zeta^{(i)}$ using the standard $W_2$ approach as in \citep{chazal2020optimal}, i.e.~without the diagonal cell $V_{k+1}$ (but with an additional centroid). We then build two simple classifiers: the predicted class assigned to a test diagram $\mu$ is $\argmin_i \{\mathrm{OT}_2(\mu, \nu^{(i)})\}$ (resp.~$(\mu,\zeta
^{(i)})$). Our $\mathrm{OT}_2$ classifier achieves a decent test accuracy of $61\%$. Advanced (kernels, deep-learning) methods in TDA reach between $72\%$ and $87\%$ of accuracy \citep[Table 1]{tda:carriere2019perslay}; but we stress that our classifier is extremely simple (we summarize a whole training class by a measure with only $k=2$ points!), showcasing that our quantizations summarize the training PDs in an informative way. More importantly, the $W_2$ classifier (with $k=3$) only achieves $50\%$ of test accuracy even though benefiting from an additional centroid, illustrating the importance of properly accounting for the diagonal as done in our approach.

\end{document}